\documentclass[11pt]{article}
\usepackage[margin=1in]{geometry}
\usepackage{amsmath,amsthm,amssymb,amsfonts}
\usepackage{mathtools}
\usepackage{bm}
\usepackage{hyperref}
\usepackage{enumitem}
\usepackage{mathrsfs}
\usepackage[dvipsnames]{xcolor}
\usepackage{verbatim}
\usepackage{booktabs}
\usepackage{graphicx}

\newtheorem{theorem}{Theorem}[section]
\newtheorem{lemma}[theorem]{Lemma}
\newtheorem{corollary}[theorem]{Corollary}

\newtheorem{definition}[theorem]{Definition}
\newtheorem{remark}[theorem]{Remark}
\newtheorem{assumption}[theorem]{Assumption}

\newcommand{\bbE}{\mathbb{E}}
\newcommand{\bbP}{\mathbb{P}}
\newcommand{\1}{\mathbf{1}}

\title{\bf Node-Private Community Detection \\ in Stochastic Block Models}

\author{Olga Klopp$^*$\quad 
Ilias Zadik$^\dagger$\\[0.5cm]
$^*$ESSEC Business School \\
$^\dagger$ Department of Statistics and Data Science, Yale University}
\date{}

\begin{document}
\maketitle

\begin{abstract}
We study community detection in stochastic block models under pure node-level differential privacy, a stringent notion that protects the participation of an individual together with all of their incident edges. This setting is substantially more challenging than edge-private community detection, since modifying a single node can affect linearly many observations. On the algorithmic side, we analyze a node-private estimator based on the exponential mechanism combined with an extension lemma, and show that exact recovery remains achievable. In the standard sparse regime with logarithmic average degree and a fixed number of communities, our results imply that a logarithmic privacy budget suffices to obtain nontrivial recovery guarantees. On the lower bound side, we show that this logarithmic scaling is in fact unavoidable: any pure node-private method must fail to achieve polynomially small exact-recovery error, or polynomially small expected mismatch, unless the privacy budget is at least of this order. Moreover, in the regime of super-logarithmic privacy budgets, our upper and lower bounds yield a matching two-term characterization of the minimax risk, with one term governed by the non-private statistical signal and the other by the privacy budget; these match up to universal constants in the exponents. Taken together, our results identify an inherent logarithmic privacy cost in node-private community detection, absent under edge differential privacy, and provide a precise rate-level characterization of the tradeoff between node privacy and SBM recovery.

\end{abstract}

\newpage 
\tableofcontents

\newpage
\section{Introduction}
Community detection is one of the central inference problems for network data \cite{HollandLaskeyLeinhardt1983,Abbe2018JMLR}. In many applications, however, the graph itself is sensitive: an edge may reveal a friendship, contact, or transaction; the presence of a vertex may reveal participation in a sensitive activity; and the inferred community label may reveal a politically, medically, or socially sensitive attribute. These considerations lead to several privacy goals. Edge-level differential privacy (DP) protects individual relationships, node-level DP protects the participation of an individual together with all incident edges, and weaker model-aware notions may instead target the latent community labels. In this paper we study the strongest graph-level notion of privacy in this hierarchy, namely pure node-level differential privacy, in the stochastic block model (SBM), and ask what statistical price it imposes on community recovery.

The SBM is the canonical benchmark for community detection \cite{HollandLaskeyLeinhardt1983,Abbe2018JMLR}. Informally, it consists of $n$ vertices partitioned into $K$ latent communities; conditional on the labels, edges are drawn independently with higher probability within communities than across communities. In the sparse regimes of interest here, these probabilities scale as $a/n$ and $b/n$. The inferential goal is to recover the latent partition from the observed graph, up to a permutation of the community labels. Throughout the paper we evaluate estimators by two standard criteria: the expected permutation-invariant mis-match ratio and the probability of exact recovery; formal definitions are deferred to Section~\ref{sec:model}.

 \paragraph{Non-private benchmark.}
Without privacy constraints, the statistical picture for the SBM is by now well-established and quite sharp. A central signal quantity is the order-$1/2$ Rényi divergence between the Bernoulli laws for within-community and across-community edges, we denote this divergence by $I$. In the sparse regime, with  the edge probabilities $a/n$ and $b/n$ and when $a,b=o(n)$ it satisfies
\[
nI=(\sqrt a-\sqrt b)^2+o(1).
\]
In the exponential-rate regime, the minimax expected mis-match risk is governed by this signal strength: for two classes the optimal exponent is $nI/2$, while for $K\ge 3$ classes it is $nI/(\beta K)$, where $\beta\ge 1$ is the class-balance parameter controlling how unequal the community sizes are allowed to be \cite{Zhang2016Minimax}. These rates are achieved by penalized likelihood methods and by computationally efficient refinements \cite{Zhang2016Minimax, Gao2017Achieving}. In the logarithmic-degree regime, exact recovery exhibits sharp threshold behavior: in the symmetric two-class model with edge probabilities $c_{\mathrm{in}}\log n/n$ and $c_{\mathrm{out}}\log n/n$, exact recovery is possible precisely when $(\sqrt{c_{\mathrm{in}}}-\sqrt{c_{\mathrm{out}}})^2>2$. More generally, in multi-community models, exact recovery is governed by Chernoff--Hellinger criteria, see, e.g., \cite{AbbeBandeiraHall2016,abbe_sandon_general_sbm}. In constant-degree regimes exact recovery is impossible, and one instead encounters detectability thresholds \cite{bordenave_lelarge_massoulie_nbt,massoulie_weak_ramanujan,mossel_neeman_sly_block_threshold}. These results provide the non-private baseline against which we measure the cost of privacy.

\paragraph{The challenge of node privacy.}
Node privacy is substantially more stringent (and correspondingly stronger) than edge privacy. Under edge-DP, changing a single edge perturbs natural likelihood-type scores by only $O(1)$, and recent works have established strong and rigorous recovery guarantees for SBM-type models under this notion \cite{MohamedNguyenVullikantiTandon2022ICML,chen2023private,NguyenVullikanti2024ICML}. In particular, leveraging this $O(1)$ edge-sensitivity, \cite{MohamedNguyenVullikantiTandon2022ICML} showed that classical stability-based algorithms require $\varepsilon=\Omega(\log n)$, whereas more refined approaches based on the Exponential Mechanism succeed with $\varepsilon=\Omega(1)$. More recently, \cite{chen2023private} further improved these guarantees, demonstrating recovery for even smaller privacy levels, allowing $\varepsilon=o(1)$ in the super-logarithmic degree regime of the SBM.

Under node-DP, by contrast, changing a single vertex may affect all of its incident edges, and therefore the same score can vary by an amount proportional to the number of vertices in the graph, namely of order $\Omega(n)$. This stark contrast has led node-private guarantees to often be regarded in the literature as essentially “impossible” for exact recovery in the SBM; see, e.g., \cite[footnote 2]{chen2023private}. More precisely, such statements reflect the expectation that exact recovery under node-DP would at least require $\varepsilon = \Omega(1)$. Moreover, the naive $O(n)$ sensitivity of natural likelihood scores raises the concern that $\varepsilon$ may even need to scale linearly with $n$ to obtain consistent estimators. In this work, we take a step toward resolving this question by systematically characterizing the trade-off between recovery guarantees in the SBM and node-level privacy.

\paragraph{Questions addressed in this paper.}
Our goal is to understand how the pure node-DP constraint changes the classical SBM benchmark. More concretely, we ask:
\begin{quote}
	\centering
	\textit{For which privacy budgets $\varepsilon>0$ can one still achieve exact recovery?}\\
	\textit{How do the statistical signal and the privacy budget jointly determine the recovery error?}
\end{quote}
These questions are already delicate in the logarithmic-degree regime, where exact recovery is near the boundary of possibility even without privacy, and where the privacy budgets of interest are much smaller than the naive linear sensitivity scale $n$

\subsection{Main contribution}

In this work, we make two main contributions for SBM community detection: an achievability result under pure node-DP, and a converse that quantifies a genuine barrier imposed by node privacy. In particular, our results show that as long as $\varepsilon$ grows \emph{very mildly with $n$}, specifically, $\varepsilon = \Omega(\log n)$, a clean and informative trade-off emerges between node privacy and SBM recovery. For instance, a perhaps intriguing finding is that $\varepsilon = \Omega(\log n)$ is necessary to achieve polynomially strong exact recovery guarantees. This scaling of $\varepsilon$ represents a significant departure from the naive sensitivity-based intuition, which would suggest that $\varepsilon = \Omega(n)$ is required to obtain any nontrivial guarantees based on likelihood scores.

More specifically, on the achievability side, we analyze an Exponential-Mechanism estimator \cite{McSherryTalwar2007} built from the homogeneous-SBM penalized likelihood score. The argument combines three key ingredients. First, we establish a ``$\gamma$-slack'' extension of the non-private risk analysis: if a labeling is within an additive score slack $\gamma$ of the optimum, then its recovery error remains controlled, with a corresponding degradation in the exponent. Second, we show that, with high probability, the graph lies in a degree envelope under which the score has controlled node sensitivity. Third, we combine this with a Lipschitz extension argument to obtain a pure node-DP estimator defined over the full graph domain. In the standard constant-SNR, logarithmic-degree regime (see Section \ref{sec:model} for a precise definition), with fixed $K$ and balance parameter $\beta$, this yields nontrivial recovery guarantees once the privacy budget is sufficiently large relative to the logarithmic sensitivity scale, i.e., $\varepsilon=\Omega(\log n).$

On the converse side, we offer a  two-point lower bound showing that any pure $\varepsilon$-node-DP mechanism has exact-recovery failure probability at least
\[
\frac{1}{1+e^{2\varepsilon}},
\]
and expected mis-match at least
\[
\frac{1}{n(1+e^{2\varepsilon})}.
\]
In particular, these bounds show that obtaining polynomially small exact-recovery failure probability or polynomially small expected mis-match already \emph{requires} $\varepsilon=\Omega(\log n)$. Therefore, the logarithmic privacy scale in our upper bound is not an artifact of the analysis; it is fundamentally imposed by node-level privacy itself.

In the regime $\varepsilon \gg \log n$ the picture becomes even cleaner.  Focusing again on the common constant SNR regime, the error of our algorithm for the upper bound decomposes in the following form
\[
\exp\{-c_1\,\mathrm{Signal}\}+\frac{1}{n}\exp\{-c_2\varepsilon\},
\]
while the lower bound has the form
\[
\exp\{-(1+o(1))\mathrm{Signal}\}+\frac{1}{n(1+e^{2\varepsilon})}.
\]
Therefore, in this regime $\varepsilon \gg \log n$, the minimax risk admits a simplified two-term structure, matching up to universal constants in the exponents: the first term governs the non-private rate, in agreement with \cite{Zhang2016Minimax}, while the second captures the statistical “price” of privacy as a function of $\varepsilon$. In particular, our bounds indicate that once $\varepsilon$ exceeds the intrinsic signal scale, the privacy term is no longer the statistical bottleneck: the private and non-private risks coincide at leading order. Thus, in that regime, privacy entails essentially no additional statistical cost. Under the standard sparse constant-SNR SBM assumptions, this signal scale is comparable to the expected degree and is often polylogarithmic in $n$.
\paragraph{Paper organization.}
The rest of the paper is organized as follows. Section~\ref{sec:model} introduces the formal SBM model, the mis-match loss, and the privacy definitions. Section~\ref{sec:exponential_mech} develops the node-private Exponential-Mechanism estimator and proves the upper bound. Section~\ref{sec:lower_bound} contains the lower bounds. The appendices collect proofs and auxiliary combinatorial arguments.

\paragraph{Notation.}
For an integer $m\ge 1$, we write $[m]:=\{1,\dots,m\}$, and $S_K$ for the permutation group on $[K]$.
We denote by $\Sigma_\beta$ the set of $\beta$-balanced labelings $\sigma:[n]\to[K]$, and by
$\Theta(n,K,a,b,\beta)$ the corresponding homogeneous SBM parameter space.
For a graph with adjacency matrix $A$, we write $\deg_A(i)$ for the degree of vertex $i$ and
$d_{\max}(A):=\max_{i\in[n]}\deg_A(i)$.

For nonnegative quantities $x$ and $y$ (possibly depending on $n$), the notation
$x\lesssim y$ means that $x\le Cy$ for some absolute constant $C>0$. Similarly, $ x \gtrsim y $ means $ y \lesssim x$ for all sufficiently large $n$ and $x\asymp y$ means both $x\lesssim y$ and $x\gtrsim y$.
We write $x\sim y$ to mean $x/y\to 1$ as $n\to\infty$.
Also, $O(\cdot)$, $\Omega(\cdot)$, $\Theta(\cdot)$ and $o(\cdot)$ are used in their standard asymptotic senses.
All logarithms are natural unless explicitly stated otherwise.

The symbol $\sim$ is also used in two standard context-dependent ways:
$A\sim \mathrm{SBM}(\sigma_0)$ means that $A$ is distributed according to the SBM with ground-truth labeling $\sigma_0$,
while $A\sim_e A'$ and $A\sim_v A'$ denote edge-adjacency and node-adjacency, respectively.

\section{Getting started: Community detection under privacy}\label{sec:setting}


\subsection{Model and accuracy}\label{sec:model}

\paragraph{SBM and parameter space.}We now formally define the stochastic block model (SBM) distribution. Without loss of generality, we assume SBM is supported on zero-diagonal, symmetric \(A\in\{0,1\}^{n\times n}\) which is the adjacency matrices of simple undirected graphs on \([n]\).

Fix integers \(n\ge 2\), \(K\ge 2\), and a balance parameter \(\beta\ge 1\). If \(K\ge 3\), also assume \(1\le \beta<\sqrt{5/3}\). Let
\[
\Sigma_\beta
:=
\Bigl\{
\sigma:[n]\to[K]:
\frac{n}{\beta K}\le |\{i:\sigma(i)=k\}| \le \frac{\beta n}{K}
\ \text{for all } k\in[K]
\Bigr\}.
\]

We now consider a ``ground-truth" labeling \(\sigma_0\) which belongs to \(\Sigma_\beta\). Conditional on \(\sigma_0\), the upper-triangular entries of \(A\) are independent and
\[
\bbP(A_{ij}=1\mid \sigma_0)=
\begin{cases}
a/n, & \sigma_0(i)=\sigma_0(j),\\
b/n, & \sigma_0(i)\neq \sigma_0(j),
\end{cases}
\qquad 1\le i<j\le n,
\]
with \(A_{ii}=0\), \(A_{ij}=A_{ji}\), and \(a>b\ge 0\).  We consider constant-SNR regime, that is we assume 
there exist constants \(0<\rho_- \le \rho_+<1\) such that
\[
\rho_- \le \frac{b}{a} \le \rho_+
\]
for all sufficiently large \(n\).
Equivalently, \(a-b=\Theta(b)\).
We write \(\Theta(n,K,a,b,\beta)\) for this homogeneous \(K\)-class \(\beta\)-balanced SBM family.
\paragraph{Accuracy.}
For labelings \(\sigma,\sigma':[n]\to[K]\), let
\[
d_H(\sigma,\sigma') := |\{i:\sigma(i)\neq \sigma'(i)\}|
\]
denote the Hamming distance. We measure recovery error by the permutation-invariant mis-match ratio
\[\label{eq:dfn_mismatch}
r(\sigma,\hat\sigma)
:=
\min_{\pi\in S_K}\frac1n\sum_{i=1}^n \1\{\sigma(i)\neq \pi(\hat\sigma(i))\}.
\]
This is the standard loss for SBM community detection. In particular, since \(nr(\sigma_0,\hat\sigma)\in\{0,1,\dots,n\}\),
\[
\bbP\bigl(r(\sigma_0,\hat\sigma)\neq 0\bigr)\le n\,\bbE\,r(\sigma_0,\hat\sigma),
\]
so an \(o(1/n)\) bound on expected mis-match implies vanishing exact-recovery failure probability.
\paragraph{Terminology.}
The two formal accuracy criteria in this paper are the expected permutation-invariant
mis-match ratio $\bbE r(\sigma_0,\hat\sigma)$ and the exact-recovery failure probability
$\Pr(r(\sigma_0,\hat\sigma)>0)$. To keep the discussion consistent, we will use the
following terminology throughout.
We say that an estimator achieves \emph{nontrivial recovery} if
\[
\sup_{\Theta(n,K,a,b,\beta)} \bbE r(\sigma_0,\hat\sigma)=o(1).
\]

We use \emph{weak recovery} as a synonym for nontrivial recovery, i.e. for the same
$o(1)$ vanishing of the expected mis-match ratio.
We say that an estimator achieves \emph{exact recovery} if
\[
\sup_{\Theta(n,K,a,b,\beta)}
\Pr\!\bigl(r(\sigma_0,\hat\sigma)>0\bigr)=o(1),
\]
equivalently, if it recovers the true labeling exactly (up to permutation of community labels)
with probability tending to one.

When a quantitative version is needed, we will say that an estimator achieves
\emph{polynomially-strong exact recovery} if there exists a constant $c>0$ such that
\[
\sup_{\Theta(n,K,a,b,\beta)}
\Pr\!\bigl(r(\sigma_0,\hat\sigma)>0\bigr)\le n^{-c}.
\]
By the same application of Markov's inequality as above, the stronger bound
\[
\sup_{\Theta(n,K,a,b,\beta)} \bbE r(\sigma_0,\hat\sigma)\le n^{-(1+c)}
\]
implies polynomially strong exact recovery.
\paragraph{Information quantity and non-private benchmark.}
Let \(I\) denote the order-\(\frac12\) R\'enyi divergence between \(\mathrm{Ber}(a/n)\) and \(\mathrm{Ber}(b/n)\):
\begin{equation}\label{eq:renyi}
I
=
-2\log\!\left(
\sqrt{\frac an\frac bn}
+
\sqrt{\left(1-\frac an\right)\left(1-\frac bn\right)}
\right).
\end{equation}
In the sparse regime \(a/n,b/n\to 0\), one has \(nI=(\sqrt a-\sqrt b)^2+o(1)\). A central benchmark for this paper is the non-private minimax theory of Zhang and Zhou~\cite{Zhang2016Minimax}, which shows that in the exponential-rate regime
\[
\inf_{\hat\sigma}\sup_{\Theta(n,K,a,b,\beta)} \bbE\,r(\sigma_0,\hat\sigma)
=
\begin{cases}
\exp\{-(1+o(1))\,nI/2\}, & K=2,\\[2mm]
\exp\{-(1+o(1))\,nI/(\beta K)\}, & K\ge 3,
\end{cases}
\]
under mild growth conditions on \(K\). Moreover, as mentioned in the introduction, this rate is attained by a regularized likelihood / homogeneous-SBM MLE procedure.


\subsection{Edge- and node-level differential privacy}

We consider two standard graph-level notions of differential privacy.

\begin{definition}[Edge-level differential privacy]
Two adjacency matrices \(A,A'\) on \([n]\) are \emph{edge-adjacent}, written \(A\sim_{\mathrm e} A'\), if they differ in exactly one unordered pair \(\{i,j\}\), \(1\le i<j\le n\). A mechanism \(M\) is \emph{\(\varepsilon\)-edge DP (edge-level)} if for all measurable sets \(S\) and all \(A\sim_{\mathrm e}A'\),
\[
\bbP(M(A)\in S)\le e^\varepsilon \bbP(M(A')\in S).
\]
\end{definition}

Edge-DP ``protects" relationships between individuals. Because changing one edge perturbs natural likelihood-type scores by \(O(1)\), it is often compatible with sparse graphs.

\begin{definition}[Node-level differential privacy]\label{def:node_privacy}
Two adjacency matrices \(A,A'\) on \([n]\) are \emph{node-adjacent}, written \(A\sim_v A'\), if they differ only in the incident edges of a single vertex \(v\in[n]\). A mechanism \(M\) is \emph{\(\varepsilon\)-node DP (node-level)} if for all measurable sets \(S\) and all \(A\sim_v A'\),
\[
\bbP(M(A)\in S)\le e^\varepsilon \bbP(M(A')\in S).
\]
\end{definition}
Node-DP ``protects" the participation of a vertex together with all incident edges and is therefore strictly stronger than edge-DP. 
In the SBM, changing one node can alter \(n-1\) edges, so in log-degree regimes the sensitivity of natural graph statistics typically grows like \(\Theta(n)\). This sensitivity growth is the main obstacle in the private analysis, and it motivates the degree-envelope and extension arguments used in Section~\ref{sec:exponential_mech}; see~\cite{KasiviswanathanNissimRaskhodnikovaSmith2013TCC,RaskhodnikovaSmith2016FOCS}. 
\paragraph{Node distance and sensitivity.}
The node-adjacency relation from Definition~\ref{def:node_privacy} induces a natural metric on the space $\mathcal G$
of graphs on the vertex set $[n]$. For $A,A' \in \mathcal G$, define the \emph{node distance}
\[
d_v(A,A')
:=
\min\Bigl\{
m\ge 0 :
\exists A^{(0)},\dots,A^{(m)} \in \mathcal G
\text{ s.t. }
A^{(0)}=A,\ A^{(m)}=A',
\text{ and }
A^{(t-1)} \sim_v A^{(t)}
\ \forall t\in[m]
\Bigr\}.
\]
Thus $d_v(A,A')$ is the minimum number of vertex-neighborhood rewirings needed to transform
$A$ into $A'$. In particular, $d_v(A,A')=1$ if and only if $A \sim_v A'$.

For a function $f:\mathcal G \to \mathbb R^d$, its \emph{node-sensitivity} is
\[
\Delta_v(f)
:=
\sup_{A \sim_v A'} \|f(A)-f(A')\|,
\]
where $\|\cdot\|$ denotes the relevant norm (absolute value in the real-valued case).
More generally, for a subset $\mathcal H \subseteq \mathcal G$, the \emph{restricted node-sensitivity}
of $f$ on $\mathcal H$ is
\[
\Delta_v(f;\mathcal H)
:=
\sup_{\substack{A,A' \in \mathcal H\\ A \sim_v A'}}
\|f(A)-f(A')\|.
\]

\subsection{Exponential Mechanism}\label{sec:EM}

The Exponential Mechanism (EM), introduced by McSherry and Talwar \cite{McSherryTalwar2007},
 is a standard differentially private procedure for selecting an output from a discrete or combinatorial candidate set when one assigns to each candidate a utility score. We describe it here in the node-DP setting. 

Given a dataset $A$, an output space $\mathcal{Y}$, and a utility function $u_A:\mathcal{Y}\to\mathbb{R}$, EM samples $y\in\mathcal{Y}$ with probability proportional to $\exp\{\varepsilon u_A(y)/(2\Delta)\}$, where 
$\Delta := \Delta_v(u)$ is the node-sensitivity of the utility under the node-adjacency relation, defined by \[
\Delta_v(u)
:=
\sup_{A \sim_v A'} \sup_{y\in\mathcal Y}
|u_A(y)-u_{A'}(y)|.
\]
 Thus higher-utility outputs are exponentially favored, while the normalization by $\Delta$ ensures the $\varepsilon$-DP constraint under the used neighborhood definition $A\sim_v A'$.

A standard consequence is that candidates whose utility is worse than the optimum by $s$ are downweighted by a factor $e^{-\varepsilon s/(2\Delta)}$, so the mechanism tends to return a near-optimizer whenever the near-optimal set is not too large.  
\subsection{Extension from a restricted domain}

A useful device in node-private graph problems is to first construct a private mechanism on a
high-probability subset of the input space where the relevant score has controlled sensitivity, and
then extend this mechanism to the full graph domain. The following fully general result of Borgs,
Chayes, Smith, and Zadik \cite{borgs_revealing_2018,borgs2018privatealgorithmsextended} gives exactly such an extension. We state it here in the present graph
setting.

Because the lemma is formulated on a metric space, we use the following metric form of privacy:
for a subset $H \subseteq \mathcal G$, a mechanism $\mathcal M : H \to \mathcal S$ is
$\varepsilon$-DP on $(H,d_v)$ if for all $A,A' \in H$ and all measurable $E \subseteq \mathcal S$,
\[
\Pr(\mathcal M(A)\in E)
\le
e^{\varepsilon d_v(A,A')}
\Pr(\mathcal M(A')\in E).
\]

\begin{lemma}[{$2\varepsilon$}-extension, \cite{borgs_revealing_2018,borgs2018privatealgorithmsextended}]
\label{prop:extension}
Let $(\mathcal G,d_v)$ be the metric space of simple undirected graphs on $[n]$ endowed with the
node distance, let $H \subseteq \mathcal G$, and let $\mathcal S$ be a measurable output space.
Suppose $\mathcal M_0 : H \to \mathcal S$ is $\varepsilon$-DP on $(H,d_v)$. Then there exists a
mechanism $\widetilde{\mathcal M} : \mathcal G \to \mathcal S$ that is $2\varepsilon$-DP on
$(\mathcal G,d_v)$ and satisfies
\[
\widetilde{\mathcal M}(A)\stackrel{d}{=} \mathcal M_0(A),
\qquad A\in H.
\]
\end{lemma}

In our upper-bound construction, Lemma \ref{prop:extension} is used  as a final transfer step.
The main work is to show that, on a high-probability degree envelope, the penalized SBM score is
Lipschitz with respect to $d_v$, so that the corresponding Exponential Mechanism is private on
that restricted domain. We then apply Lemma \ref{prop:extension} with $H=\mathcal G_C$ to obtain
a full-domain node-private estimator.

\section{The upper bound: a node-private estimator via the Exponential Mechanism}\label{sec:exponential_mech}

We now state the node-private estimator and its performance guarantee. The estimator is obtained by applying the Exponential Mechanism (EM) to the homogeneous-SBM penalized likelihood score on a high-probability bounded-degree envelope, and then extending the resulting restricted-domain mechanism to the full graph domain.

\subsection{The estimator}\label{sec:estimator}
We take as candidate set the balanced labelings $\Sigma_\beta$, and as utility the penalized score $u_A(\sigma)=T_A(\sigma)$, introduced in \cite{Zhang2016Minimax} and defined for each labeling $\sigma:[n]\to[K]$ by
\begin{equation}\label{eq:T}
	T_A(\sigma) = \sum_{i<j} A_{ij}\,\1\{\sigma(i)=\sigma(j)\}
	\;-
	\lambda \sum_{i<j} \1\{\sigma(i)=\sigma(j)\}.
\end{equation}
Equivalently,
\[
T_A(\sigma)=\sum_{i<j}(A_{ij}-\lambda)\,\1\{\sigma(i)=\sigma(j)\}.
\]
Thus $T_A(\sigma)$ rewards observed within-community edges and penalizes excessively large within-community blocks.
As in \cite{Zhang2016Minimax}, we choose $\lambda$ through the Chernoff tilt
\begin{equation}\label{eq:tstar}
	t^\star = \frac12 \log\!\left(\frac{a(1-b/n)}{b(1-a/n)}\right).
\end{equation}
Specifically, for $K=2$, we set
\begin{equation}\label{eq:lambdaK2}
	\lambda = -\frac{1}{2t^\star}\log\!\Big(\frac{a}{n}e^{-t^\star}+1-\frac{a}{n}\Big)
	+\frac{1}{2t^\star}\log\!\Big(\frac{b}{n}e^{t^\star}+1-\frac{b}{n}\Big).
\end{equation}
and for $K\ge 3$, we fix any $w\in[0,1]$ and set
\begin{equation}\label{eq:lambdaKge3}
	\lambda = -\frac{w}{t^\star}\log\!\Big(\frac{a}{n}e^{-t^\star}+1-\frac{a}{n}\Big)
	+\frac{1-w}{t^\star}\log\!\Big(\frac{b}{n}e^{t^\star}+1-\frac{b}{n}\Big).
\end{equation}
\begin{remark}
With this choice, likelihood comparisons admit the sharp Chernoff bounds governed by the order-$\frac12$ R\'enyi divergence $I$. Note that under  constant-SNR regime and \(a,b=o(n)\), one has 
\[
t^\star=\Theta(1), \qquad nI\asymp a,
\]
with constants depending only on \(\rho_-,\rho_+\).
\end{remark}

A key challenge for node-DP is the large worst-case node-sensitivity of $T_A(\sigma)$ on the full
graph space: rewiring one node can change linearly many edge terms. To avoid calibrating the
Exponential Mechanism to this global worst case, we first restrict attention to a high-probability
degree envelope on which the score is Lipschitz with respect to the node distance. We then apply
Lemma \ref{prop:extension}  in order to extend the resulting restricted-domain
private mechanism to all graphs.
 For a constant $C>0$, we define the degree envelope
\begin{equation}\label{eq:Ec}
	\mathcal G_C
	:=
	\Big\{A:\ d_{\max}(A):=\max_i \deg_A(i)\le C\max\{a,\log n\}\Big\},
\end{equation}
and set
\begin{equation}\label{eq:Delta}
	\Delta_a := 2C\max\{a,\log n\}.
\end{equation}
By Lemma~\ref{lem:degree-envelope}, we prove that for some $C>0$, the SBM instance lies in $\mathcal G_C$ with high probability, and by Lemma~\ref{lem:sensitivity}, the restricted node-sensitivity of $T_A$ on $\mathcal G_C$ is at most $\Delta_a$. Observe that $\Delta_a$ can be, in principle, much smaller than $\Theta(n)$.

\paragraph{The node-DP estimator}
We are now in a position to describe our $\varepsilon$-node-DP estimator defined with input an $n$-node undirected graph $A$. The estimator is constructed in two steps.

First, we define the estimator when $A \in \mathcal G_C$ where $C>0$ is a constant defined in Lemma~\ref{lem:degree-envelope}. In that case we run the EM with utility $T_A$, sensitivity $\Delta_{a}$ and $\varepsilon_0: =\varepsilon/2$, i.e., output the labeling based on
\begin{equation}\label{eq:em}
	\mathbb P(\hat\sigma=\sigma\mid A)
	=
	\frac{
		\exp\!\left(\frac{\varepsilon_0}{2\Delta_a}T_A(\sigma)\right)\mathbf 1\{\sigma\in\Sigma_\beta\}
	}{
		\sum_{\tau\in\Sigma_\beta}
		\exp\!\left(\frac{\varepsilon_0}{2\Delta_a}T_A(\tau)\right)
	}.
\end{equation}
Second, we extend the estimator from $\mathcal G_C$ to the whole space of $n$-vertex undirected graphs, using the Lipschitz extension technique described in Lemma \ref{prop:extension}.
\begin{lemma}[Privacy of the restricted-domain and full-domain estimators]
\label{lem:em-private}
Set $\varepsilon_0:=\varepsilon/2$. Define the restricted-domain mechanism
$\hat\sigma:\mathcal G_C\to\Sigma_\beta$ by \eqref{eq:em}. Then $\hat\sigma$ is
$\varepsilon_0$-DP on $(\mathcal G_C,d_v)$. Consequently, by Lemma
\ref{prop:extension}, there exists a mechanism
\[
\hat\sigma^f:\mathcal G\to\Sigma_\beta
\]
that is $\varepsilon$-node-DP on $\mathcal G$ and satisfies
\[
\hat\sigma^f(A)\stackrel{d}{=}\hat\sigma(A),
\qquad A\in\mathcal G_C.
\]
\end{lemma}

\begin{proof}
By Lemma \ref{lem:sensitivity}, for all $A,A'\in\mathcal G_C$ and all
$\sigma\in\Sigma_\beta$,
\[
|T_A(\sigma)-T_{A'}(\sigma)|\le \Delta_a d_v(A,A').
\]
Hence
\[
\exp\!\Bigl(-\frac{\varepsilon_0}{2}d_v(A,A')\Bigr)
\le
\exp\!\Bigl(\frac{\varepsilon_0}{2\Delta_a}(T_A(\sigma)-T_{A'}(\sigma))\Bigr)
\le
\exp\!\Bigl(\frac{\varepsilon_0}{2}d_v(A,A')\Bigr).
\]
Let
\[
Z_A:=\sum_{\tau\in\Sigma_\beta}
\exp\!\Bigl(\frac{\varepsilon_0}{2\Delta_a}T_A(\tau)\Bigr).
\]
Applying the previous bound termwise and summing over $\tau$ yields
\[
e^{-\varepsilon_0 d_v(A,A')/2}Z_{A'}
\le
Z_A
\le
e^{\varepsilon_0 d_v(A,A')/2}Z_{A'}.
\]
Therefore, for every measurable $S\subseteq\Sigma_\beta$,
\[
\Pr(\hat\sigma(A)\in S)
=
\sum_{\sigma\in S}
\frac{\exp\!\bigl(\frac{\varepsilon_0}{2\Delta_a}T_A(\sigma)\bigr)}{Z_A}
\le
e^{\varepsilon_0 d_v(A,A')}
\sum_{\sigma\in S}
\frac{\exp\!\bigl(\frac{\varepsilon_0}{2\Delta_a}T_{A'}(\sigma)\bigr)}{Z_{A'}}
=
e^{\varepsilon_0 d_v(A,A')}
\Pr(\hat\sigma(A')\in S).
\]
Thus $\hat\sigma$ is $\varepsilon_0$-DP on $(\mathcal G_C,d_v)$. This is the standard privacy
guarantee of the Exponential Mechanism; see McSherry and Talwar~\cite{McSherryTalwar2007} and, for textbook background, Dwork and Roth \cite{DworkRoth2014}.

The second claim follows directly from Lemma \ref{prop:extension} applied with
$H=\mathcal G_C$.
\end{proof}

\subsection{The main theorem: the rate of the estimator}
We now present our main result.
  Let
	\[
	\mathrm{Signal}:=
	\begin{cases}
		nI/2,& K=2,\\[1mm]
		nI/(\beta K),& K\ge 3,
	\end{cases}
	\]

We first describe the parameter assumptions our result needs.
\begin{assumption}[Signal--entropy separation]\label{ass:strong}
	There exists a  sufficiently large absolute constant $C_s>0$ such that for all sufficiently large $n$, 
    \begin{equation}\label{eq:signal-entropy}
\mathrm{Signal} \ge C_s \log(nK).
	\end{equation}
\end{assumption}
Assumption~\ref{ass:strong} says that the SBM signal must be strong enough to dominate the combinatorial entropy of incorrect labelings. In the sparse constant-SNR regime with fixed $K$ and $\beta$, this is essentially the logarithmic-degree condition $a \gtrsim \log n$.

\begin{assumption}[Mild growth of $K$]\label{ass:mildK}
	There exists an absolute constant $C_{\mathrm{mg}}>0$ such that, for all sufficiently large $n$,
	\[
	\log(nK)\ \ge\ C_{\mathrm{mg}}\, K\log K.
	\]
	Equivalently,
	\[
	 K \ \le\ C_{\mathrm{mg}}^{-1}\,\frac{\log(nK)}{\log K}.
	\]
\end{assumption}
Assumption~\ref{ass:mildK} says that the number of communities may grow with $n$, but not so fast that the combinatorial overhead in $K$ overwhelms the exponential-rate analysis. This is essentially $K\log K \lesssim \log n$, i.e. $K$ is at most nearly logarithmic in $n$ up to a $\log\log n$ factor.

We can now state the main result of the section. In the statement we use the notation introduced in Section \ref{sec:estimator}.

\begin{theorem}[$\varepsilon$-node-DP risk bound]\label{thm:private_rate}
	Assume our parameters satisfy Assumptions~\ref{ass:strong} and~\ref{ass:mildK}. Let $A\mapsto \hat\sigma_f$ be the $\varepsilon$-node-DP algorithm defined in Section \ref{sec:estimator}. 
    Let
	\[
	B:=C_0\frac{K\log(nK)}{nI},
	\qquad
	\eta:=\frac{\varepsilon}{4\Delta_a},
	\qquad
	\gamma_0:=\eta-B.
	\]

	
	 There exist absolute constants $c_0,c_1,C_0,C_1>0$ such that the following holds.\\ 
	Assume $\gamma_0>0$ or equivalently \begin{align}\label{eq:feasib}
\frac{\varepsilon}{4\Delta_a} \geq C_0\frac{K\log(nK)}{nI}.
	\end{align} Then for every $0<\alpha<\frac12$, if we set
	\begin{equation}\label{eq:sstar}
		s^\star=s^\star(\alpha):=\frac{C_1\log(nK)+\log(4/\alpha)}{\gamma_0}
	\end{equation}
	then it holds
	\[
	\mathbb E\,r(\sigma_0,\hat\sigma_f)
	\le
	\exp\!\Big\{-(1+o(1))\,\mathrm{Signal}+t^\star s^\star\Big\}
	+\alpha+e^{-c_0 \mathrm{Signal}}+e^{-c_1(a+\log n)}.
	\]
	Moreover, if $\eta\ge 2B$ we choose for any fixed $c_3>0,$
	\[
	\alpha:=\frac{1}{nK}e^{-c_3\varepsilon/2}
	\]
	in which case
	\[
	s^\star \lesssim \Delta_a\,\frac{\log(nK)}{\varepsilon}+\Delta_a,
	\]
	and hence 
   
	\begin{equation}\label{eq:cor_simplified}
		\mathbb E\,r(\sigma_0,\hat\sigma_f)
		\le
		\exp\!\Big\{-(1+o(1))\,\mathrm{Signal}
		+\mathcal O\!\Bigl( \Delta_a \frac{\log(nK)}{\varepsilon}\Bigr)
		+\mathcal O(\Delta_a)\Big\}
		+\frac{1}{nK}e^{-\Omega(\varepsilon)}
		+e^{-c_0 \mathrm{Signal}}
		+e^{-c_1(a+\log n)}.
	\end{equation}
\end{theorem}

\paragraph{Proof sketch.}
Conditional on $A\in\mathcal G_C$ the law of $\hat\sigma_f$ coincides with that of the restricted-domain EM. Lemma~\ref{lem:em-utility} shows that this mechanism
returns a labeling whose score is within $s^\star$ of $\max_{\tau\in\Sigma_\beta} T_A(\tau)$ with
probability at least $1-\alpha-e^{-c_0\mathrm{Signal}}$. We then introduce an auxiliary estimator
that equals $\hat\sigma_f$ on this good event and otherwise is a maximizer of $T_A$; this
estimator is always within score slack $s^\star$ of optimality. Lemma~\ref{lem:risk-with-score-slack}
therefore yields the desired exponential bound, with loss $t^\star s^\star$ in the exponent, and the
bad event contributes the additive term
$\alpha + e^{-c_0\mathrm{Signal}} + e^{-10(a+\log n)}$.

We now simplify our parameter assumptions in a natural, well-studied regime.
\begin{corollary}[Logarithmic privacy scale in the fixed-$K$, constant-SNR, log-degree regime]
\label{cor:log-privacy-fixed-K}
Under the assumptions of Theorem~\ref{thm:private_rate}, suppose in addition that
$
a,b=o(n)$, $a \gtrsim \log n$
and that \(K=O(1)\). Then the feasibility condition
\[
\varepsilon \gtrsim \Delta_a \frac{K\log(nK)}{nI}
\]
reduces to
\[
\varepsilon \gtrsim \log n.
\]
Consequently, in the fixed-$K$, constant-SNR, log-degree regime, a logarithmic privacy budget in \(n\) is sufficient for the conclusion of Theorem~\ref{thm:private_rate}  to hold.
\end{corollary}

\begin{proof}
Under the standing constant-SNR assumption and the sparse regime \(a,b=o(n)\), one has
\[
nI \asymp a.
\]
Moreover, since
\[
\Delta_a = 2C\max\{a,\log n\},
\]
the condition \(a \gtrsim \log n\) implies
\[
\Delta_a \asymp a.
\]
Hence
\[
\varepsilon \gtrsim \Delta_a \frac{K\log(nK)}{nI}
\asymp K\log(nK).
\]
If \(K\) is fixed, then \(K\log(nK)\asymp \log n\), which proves the claim.
\end{proof}

\begin{corollary}[Two-term upper bound for large privacy budgets]
\label{cor:large_privacy_two_term}
Under the assumptions of Corollary~\ref{cor:log-privacy-fixed-K},  there exist universal constants $L,c_{\mathrm{sig}},c_{\mathrm{priv}}>0$ such that,
for all sufficiently large $n$, if
\[
\varepsilon \ge L \log n,
\]
then
\[
\sup_{\Theta(n,K,a,b,\beta)}
E r(\sigma_0,\hat\sigma_f)
\;\le\;
\exp\{-c_{\mathrm{sig}}\,\mathrm{Signal}\}
\;+\;
\exp\{-c_{\mathrm{priv}}\varepsilon\}.
\]

\end{corollary}

\begin{proof}
Under the constant-SNR assumption and the sparse regime $a,b=o(n)$, one has
$nI \asymp a$. Since $a \gtrsim \log n$ and
\[
\Delta_a = 2C\max\{a,\log n\},
\]
it follows that $\Delta_a \asymp a$. Because $K=O(1)$ and $\beta=O(1)$, we also have
\[
\mathrm{Signal} \asymp a.
\]
Therefore
\[
B = C_0\frac{K\log(nK)}{nI} \asymp \frac{\log(nK)}{a},
\qquad
\eta = \frac{\varepsilon}{4
\Delta_a} \asymp \frac{\varepsilon}{a}.
\]

Choose a sufficiently small universal constant $\kappa>0$ and set
\[
\alpha := \frac{1}{nK}e^{-\kappa\varepsilon}.
\]
If $\varepsilon \ge L\log n$ with $L$ large enough, then $\eta \ge 2B$, and hence
\[
\gamma_0=\eta-B \asymp \frac{\varepsilon}{a}.
\]
By Theorem~\ref{thm:private_rate},
\[
s^\star
=
\frac{C_1\log(nK)+\log(4/\alpha)}{\gamma_0}
\;\lesssim\;
\frac{a\log(nK)}{\varepsilon}+\kappa a.
\]
Since $t^\star=\Theta(1)$ in the constant-SNR sparse regime, this gives
\[
t^\star s^\star
\;\lesssim\;
\frac{a\log(nK)}{\varepsilon}+\kappa a.
\]
Because $\mathrm{Signal}\asymp a$, by taking $L$ large enough and then $\kappa$ small enough,
we obtain
\[
t^\star s^\star \le \frac{1}{2}\,\mathrm{Signal}
\]
for all sufficiently large $n$. Hence the leading term in Theorem~\ref{thm:private_rate} satisfies
\[
\exp\{-(1+o(1))\mathrm{Signal}+t^\star s^\star\}
\le
\exp\{-c_{\mathrm{sig}}\mathrm{Signal}\}
\]
for some universal constant $c_{\mathrm{sig}}>0$.

For the remaining terms, we have 
\[
\alpha=\frac{1}{nK}e^{-\kappa\varepsilon},
\qquad
e^{-c_0\mathrm{Signal}}
\le
e^{-c'_{\mathrm{sig}}\mathrm{Signal}},
\qquad
e^{-10(a+\log n)}
\le
e^{-c''_{\mathrm{sig}}\mathrm{Signal}},
\]
since $a\asymp \mathrm{Signal}$. Absorbing the last two terms into the signal-exponential term and $\varepsilon \ge L \log n$
proves the claim.
\end{proof}
\section{The lower bound}\label{sec:lower_bound}

We now present our node-privacy lower bound, clarifying further the trade-off between privacy and SBM community detection. For fixed $(a,b)$, each parameter $\theta\in\Theta(n,K,a,b,\beta)$ is determined by a balanced
labeling only up to permutation of the community labels. For each $\theta\in\Theta(n,K,a,b,\beta)$,
fix an arbitrary representative labeling $\sigma_\theta\in\Sigma_\beta$. Since both the SBM law and
the loss $r(\sigma,\hat\sigma)$ are invariant under relabeling, all quantities below are independent
of the chosen representative.
\begin{theorem}\label{thm:lower_bound}[Two-point minimax lower bounds under node-DP]
Fix $K \ge 2$ and $\beta \ge 1$. Assume $n/(\beta K)\ge 2$ so that each class contains at least
two vertices. Fix $(a,b)$, and let
\[
\Theta := \Theta(n,K,a,b,\beta)
\]
be the corresponding homogeneous SBM family.

Let $M$ be any $\varepsilon$-node-DP mechanism (with respect to node-adjacency), and let
\[
\hat\sigma = M(A)\in [K]^n
\]
be its output on input adjacency matrix $A$.

For each $\theta\in\Theta$, let $\sigma_\theta\in\Sigma_\beta$ be a representative ground-truth
labeling corresponding to $\theta$, and let $\mathbb P_\theta$ and $\mathbb E_\theta$ denote
probability and expectation under $\theta$.

Define the exact-recovery failure probability
\[
\delta(M):=\sup_{\theta\in\Theta}\mathbb P_\theta\!\left(r(\sigma_\theta,\hat\sigma)>0\right).
\]
Then
\[
\delta(M)\ge \frac{1}{1+e^{2\varepsilon}} \ge \frac12 e^{-2\varepsilon}.
\]
Consequently,
\begin{equation}\label{Eq:mismatch_pr_lb}
\inf_{M\ \varepsilon\text{-node-DP}}
\sup_{\theta\in\Theta}
\mathbb P_\theta\!\left(r(\sigma_\theta,M(A))>0\right)
\ge
\frac{1}{1+e^{2\varepsilon}}.
\end{equation}

Moreover, define the expected mis-match minimax risk
\[
R(M):=\sup_{\theta\in\Theta}\mathbb E_\theta\!\left[r(\sigma_\theta,\hat\sigma)\right].
\]
Then
\[
R(M)\ge \frac{1}{n(1+e^{2\varepsilon})}\ge \frac{1}{2n}e^{-2\varepsilon},
\]
and therefore
\begin{equation}\label{eq:Er-lb}
\inf_{M\ \varepsilon\text{-node-DP}}
\sup_{\theta\in\Theta}
\mathbb E_\theta\!\left[r(\sigma_\theta,M(A))\right]
\ge
\frac{1}{n(1+e^{2\varepsilon})}.
\end{equation}
\end{theorem}

\begin{proof}
	\leavevmode\paragraph{Step 1: Pick a hard pair of balanced labelings.}
    Fix any $\theta\in\Theta$, and let $\sigma:=\sigma_\theta\in\Sigma_\beta$ be a representative
ground-truth labeling.
	 Choose two vertices $u\neq v$ that belong to two different
	communities under $\sigma$, say $\sigma(u)=k$ and $\sigma(v)=\ell$ with $k\neq \ell$.
	Define $\sigma'\in[K]^n$ by swapping the labels of $u$ and $v$:
	\[
	\sigma'(u)=\ell,\qquad \sigma'(v)=k,\qquad \sigma'(i)=\sigma(i)\ \text{ for }i\notin\{u,v\}.
	\]
	Then $\sigma'$ has exactly the same community sizes as $\sigma$, hence $\sigma'\in\Sigma_\beta$. Let $\theta'\in \Theta$ be the SBM parameter corresponding to $\sigma'$.
	
	Let $S_K$ be the permutation group on $[K]$, and write the orbit
	$\Gamma(\sigma):=\{\pi\circ\sigma:\pi\in S_K\}$.
	Define the \emph{exact recovery} events
	\[
	E_\sigma \;:=\;\{\hat\sigma\in\Gamma(\sigma)\}
	\qquad\text{and}\qquad
	E_{\sigma'} \;:=\;\{\hat\sigma\in\Gamma(\sigma')\}.
	\]
	By definition of the mis-match ratio, $E_\sigma=\{r(\sigma,\hat\sigma)=0\}$ and similarly for
	$\sigma'$.
	
	We claim that these target sets are disjoint:
	\[
	\Gamma(\sigma)\cap\Gamma(\sigma')=\varnothing.
	\]
	Indeed, because $n/(\beta K)\ge 2$, community $k$ contains some vertex $w\neq u$ and
	community $\ell$ contains some vertex $w'\neq v$.
	If $\sigma'=\pi\circ\sigma$ for some $\pi\in S_K$, then
	$\sigma'(w)=\sigma(w)=k$ implies $\pi(k)=k$, and $\sigma'(w')=\sigma(w')=\ell$ implies
	$\pi(\ell)=\ell$. But then $\sigma'(u)=\pi(\sigma(u))=\pi(k)=k$, contradicting $\sigma'(u)=\ell$.
	Thus $\sigma'\notin\Gamma(\sigma)$ and the orbits are disjoint, which implies
	\[
	E_\sigma\cap E_{\sigma'}=\varnothing.
	\]
	
	\leavevmode\paragraph{Step 2: Couple the two SBM graph distributions with node-distance $\le 2$.}
	Let $\mu_\sigma$ (resp.\ $\mu_{\sigma'}$) denote the distribution of $A$ under
	$\mathrm{SBM}(\sigma)$ (resp.\ $\mathrm{SBM}(\sigma')$).
	Construct a coupling $(A,A')$ of $\mu_\sigma$ and $\mu_{\sigma'}$ as follows:
	
	\begin{itemize}
		\item For all pairs $\{i,j\}$ with 
		$\{i,j\}\cap\{u,v\}=\varnothing$, sample $A_{ij}$ once according
		to the SBM edge probability under $\sigma$, and set $A'_{ij}=A_{ij}$.
		(This is valid because $\sigma$ and $\sigma'$ coincide on all vertices outside $\{u,v\}$, so the
		edge probabilities for these pairs are identical under $\sigma$ and $\sigma'$.)
		\item For all edges incident to $u$ or $v$ (including $\{u,v\}$), sample $(A_{ij})$ according to the
		correct SBM probabilities under $\sigma$ and sample $(A'_{ij})$ according to the correct SBM
		probabilities under $\sigma'$, independently of the first step.
	\end{itemize}
	
	Then $A\sim\mu_\sigma$, $A'\sim\mu_{\sigma'}$, and $A$ and $A'$ differ \emph{only} on edges incident
	to $u$ and/or $v$. Therefore, in the node metric $d_v$ (minimum number of vertex-neighborhood
	rewirings), we have
	\[
	d_v(A,A')\le 2 \qquad \text{almost surely under the coupling.}
	\]
	
	\paragraph{Step 3: Use group privacy to compare output probabilities under $\sigma$ and $\sigma'$.}
	Since $M$ is $\varepsilon$-node-DP, group privacy implies that for any two graphs $G,G'$
	with $d_v(G,G')\le 2$ and any measurable output event $S$,
	\[
	\mathbb P(M(G)\in S)\ \le\ e^{2\varepsilon}\, \mathbb P(M(G')\in S).
	\]
	Applying this pointwise to $(A,A')$ in the above coupling (where $d_v(A,A')\le 2$ a.s.) and taking
	expectations over the coupling yields, for all measurable $S$,
	\[
	\mathbb P_{\theta}\!\big(M(A)\in S\big)
	\;\le\;
	e^{2\varepsilon}\,
	\mathbb P_{\theta'}\!\big(M(A)\in S\big).
	\]
	Equivalently,
	\[
	\mathbb P_{\theta'}\!\big(M(A)\in S\big)
	\;\ge\;
	e^{-2\varepsilon}\,
	\mathbb P_{\theta}\!\big(M(A)\in S\big).
	\]
	
	\paragraph{Step 4: Plug in the disjoint exact-recovery events and derive the lower bound on $\delta$.}
	Apply the last inequality with $S=\Gamma(\sigma)$, i.e.\ $S=E_\sigma$:
	\[
	\mathbb P_{\theta'}(E_\sigma)\ \ge\ e^{-2\varepsilon}\, \mathbb P_{\theta}(E_\sigma).
	\]
	Because $E_\sigma\cap E_{\sigma'}=\varnothing$, under $\sigma'$ we have
	\[
	\mathbb P_{\theta'}(E_\sigma)\ \le\ 1-\mathbb P_{\theta'}(E_{\sigma'}).
	\]
	Combine the two displays:
	\[
	1-\mathbb P_{\theta'}(E_{\sigma'})\ \ge\ e^{-2\varepsilon}\, \mathbb P_{\theta}(E_\sigma).
	\]
	Define the failure probabilities
	\[
	\delta_\sigma := 1-\mathbb P_{\theta}(E_\sigma)
	= \mathbb P_\theta\!\big(r(\sigma,\hat\sigma)>0\big),
	\qquad
	\delta_{\sigma'} := 1-\mathbb P_{\theta'}(E_{\sigma'})
	= \mathbb P_{\theta'}\!\big(r(\sigma',\hat\sigma)>0\big).
	\]
	Then the previous inequality becomes
	\[
	\delta_{\sigma'} \ \ge\ e^{-2\varepsilon}(1-\delta_\sigma).
	\]
	Now set $\delta := \max\{\delta_\sigma,\delta_{\sigma'}\}$. Since $1-\delta_\sigma\ge 1-\delta$,
	\[
	\delta \;\ge\; \delta_{\sigma'} \;\ge\; e^{-2\varepsilon}(1-\delta).
	\]
	Rearranging,
	\[
	\delta(1+e^{-2\varepsilon})\ \ge\ e^{-2\varepsilon}
	\qquad\Longrightarrow\qquad
	\delta\ \ge\ \frac{e^{-2\varepsilon}}{1+e^{-2\varepsilon}}
	\ =\ \frac{1}{1+e^{2\varepsilon}}.
	\]
	Finally, since $1+e^{2\varepsilon}\le 2e^{2\varepsilon}$, we also have
	\[
	\delta\ \ge\ \frac{1}{1+e^{2\varepsilon}}\ \ge\ \frac12 e^{-2\varepsilon}.
	\]
Since for this fixed mechanism $M$ we have found two parameters $\theta,\theta'\in\Theta$ such
that
\[
\max\Bigl\{
\mathbb P_\theta(r(\sigma_\theta,\hat\sigma)>0),
\mathbb P_{\theta'}(r(\sigma_{\theta'},\hat\sigma)>0)
\Bigr\}
\ge \frac{1}{1+e^{2\varepsilon}},
\]
it follows that
\[
\delta(M)
=
\sup_{\vartheta\in\Theta}
\mathbb P_\vartheta\!\left(r(\sigma_\vartheta,\hat\sigma)>0\right)
\ge
\frac{1}{1+e^{2\varepsilon}}.
\]

    \paragraph{Step 5: Convert the failure lower bound into an $\mathbb E[r]$ lower bound.}
For each $\theta\in\Theta$, since
\[
n\,r(\sigma_\theta,\hat\sigma)\in\{0,1,\dots,n\},
\]
we have
\[
r(\sigma_\theta,\hat\sigma)\ge \frac1n \mathbf 1\{r(\sigma_\theta,\hat\sigma)>0\}.
\]
Taking expectation under $\theta$ and then the supremum over $\theta\in\Theta$ gives
\[
R(M)
=
\sup_{\theta\in\Theta}\mathbb E_\theta[r(\sigma_\theta,\hat\sigma)]
\ge
\frac1n
\sup_{\theta\in\Theta}\mathbb P_\theta\!\left(r(\sigma_\theta,\hat\sigma)>0\right)
=
\frac{\delta(M)}{n}.
\]
Combining this with the bound proved in Step~4,
\[
\delta(M)\ge \frac{1}{1+e^{2\varepsilon}},
\]
yields
\[
R(M)\ge \frac{1}{n(1+e^{2\varepsilon})}.
\]
Taking the infimum over all $\varepsilon$-node-DP mechanisms $M$ concludes the minimax
lower bound for the expected mis-match risk.
\end{proof}
\begin{remark}[Uniformity over larger parameter classes]\label{rq:larger_class}
Although Theorem~\ref{thm:lower_bound} is stated for the fixed-$(a,b)$ class
$\Theta(n,K,a,b,\beta)$, the proof does not use the numerical values of $a$ and $b$.
It only relies on the existence of two balanced labelings $\sigma,\sigma'\in\Sigma_\beta$
that differ by swapping two vertices from different communities, together with the fact
that the corresponding SBM laws can be coupled so that the sampled graphs differ only on
the neighborhoods of those two vertices. Consequently, the same lower bound holds uniformly
over any larger parameter class obtained by allowing $(a,b)$ to vary.

More precisely, if $\mathcal A_n$ is any nonempty collection of admissible pairs $(a,b)$ and
\[
\widetilde{\Theta}(n,K,\beta;\mathcal A_n)
:=
\bigcup_{(a,b)\in \mathcal A_n}\Theta(n,K,a,b,\beta),
\]
then
\[
\inf_{M\ \varepsilon\text{-node-DP}}
\sup_{\theta\in \widetilde{\Theta}(n,K,\beta;\mathcal A_n)}
\mathbb P_\theta\!\left(r(\sigma_\theta,M(A))>0\right)
\ge \frac{1}{1+e^{2\varepsilon}},
\]
and similarly
\[
\inf_{M\ \varepsilon\text{-node-DP}}
\sup_{\theta\in \widetilde{\Theta}(n,K,\beta;\mathcal A_n)}
\mathbb E_\theta\!\left[r(\sigma_\theta,M(A))\right]
\ge \frac{1}{n(1+e^{2\varepsilon})}.
\]
Indeed, for every fixed $(a,b)\in \mathcal A_n$, one has
$\Theta(n,K,a,b,\beta)\subseteq \widetilde{\Theta}(n,K,\beta;\mathcal A_n)$, so the
larger-class lower bound follows immediately from Theorem~\ref{thm:lower_bound}.
\end{remark}


Remark \ref{rq:larger_class} shows that Theorem \ref{thm:lower_bound} continues to hold uniformly over broad unions of SBM parameter classes. To interpret this lower bound in a concrete and classical regime, it is natural to compare it with the non-private benchmark in the logarithmic-degree setting near the exact-recovery threshold. In the symmetric two-community model with $a = c_{\mathrm{in}} \log n$ and $b = c_{\mathrm{out}} \log n$, for constants $0 < c_{\mathrm{out}} < c_{\mathrm{in}}$, exact recovery is possible precisely when  $(\sqrt{c_{\mathrm{in}}} - \sqrt{c_{\mathrm{out}}})^2 > 2$ \cite{AbbeBandeiraHall2016}. In this regime, the non-private minimax expected mis-match risk, by \cite{Zhang2016Minimax}, is of the order $\exp(-\mathrm{Signal})$, which for $K=2$ and under standard sparse asymptotics becomes \[\exp\{\mathrm{-Signal}\}=n^{-(\sqrt{c_{\mathrm{in}}}-\sqrt{c_{\mathrm{out}}})^2/2+o(1)}\] 
By Markov’s inequality, this implies that the exact-recovery failure probability is at most
\[n\exp\{\mathrm{-Signal}\}=n^{1-(\sqrt{c_{\mathrm{in}}}-\sqrt{c_{\mathrm{out}}})^2/2+o(1)}=n^{-\Omega(1)}\] 
whenever the threshold condition holds. This makes polynomially small exact-recovery failure a natural benchmark in this regime. The next corollary shows that, under pure node-level differential privacy, achieving such a guarantee already requires $\varepsilon=\Omega(\log n)$.
\begin{corollary}
	\label{cor:logn-necessary}
	For any constant $c>0$:
	\begin{itemize}
		\item If $M$ achieves polynomially small exact-recovery failure,
		\[
		\delta(M)\ \le\ n^{-c},
		\]
		then necessarily
		\[
		\varepsilon\ \ge\ \frac12\log(n^c-1)\ \ge\ \frac{c}{2}\log n-\frac12\log 2,
		\]
		so $\varepsilon=\Omega(\log n)$.
		\item If $M$ achieves polynomially small expected mismatch,
		\[
		R(M)\ \le\ n^{-(1+c)},
		\]
		then necessarily
		\[
		\varepsilon\ \ge\ \frac12\log(n^c-1)\ \ge\ \frac{c}{2}\log n-\frac12\log 2,
		\]
		so again $\varepsilon=\Omega(\log n)$.
	\end{itemize}
\end{corollary}
\subsection{Constant SNR regime}
Recall from Corollary \ref{cor:large_privacy_two_term}, in the constant SNR regime and under $\varepsilon \gg \log n,$ \[
\sup_{\Theta(n,K,a,b,\beta)}
\mathbb{E} r(\sigma_0,\hat\sigma_f)
\;\le\;
\exp\{-c_{\mathrm{sig}}\,\mathrm{Signal}\}
\;+\;
\exp\{-c_{\mathrm{priv}}\varepsilon\}.
\]
While Theorem \ref{thm:lower_bound} implies that  under $\varepsilon \gg \log n$,
 \[
\inf_{M\ \varepsilon\text{-node-DP}}\sup_{\Theta(n,K,a,b,\beta)}
\mathbb{E} r(\sigma_0,\hat\sigma)
\;\ge\;
n^{-1}\exp\{-\Theta(\varepsilon)\}=\exp\{-\Theta(\varepsilon)\}.
\]
On the other hand, since the class of $\varepsilon$-node-DP estimators is a subset of all estimators,  the non-private minimax lower bound from  \cite{Zhang2016Minimax} implies
 \begin{align}\label{eq:non_priv_LB}
\inf_{M\ \varepsilon\text{-node-DP}}\sup_{\Theta(n,K,a,b,\beta)}
\mathbb{E} r(\sigma_0,\hat\sigma)
\;\ge\;\exp\{-\Theta(\mathrm{Signal})\}.
\end{align}leading to the following characterization of the minimax-rate in the constant SNR regime.

\begin{corollary}[Two-sided minimax bound in the constant-SNR regime]\label{cor:rate_cnst}
Let
\[
\mathcal R_\varepsilon(n,K,a,b,\beta)
:=
\inf_{M:\,\varepsilon\text{-node-DP}}
\sup_{\Theta(n,K,a,b,\beta)}
\mathbb E r(\sigma_0,M(A)).
\]
Under the assumptions of Corollary~\ref{cor:log-privacy-fixed-K}, there exist universal constants
\[
L,\; c_{\mathrm{sig}},\; c_{\mathrm{priv}},\; C_{\mathrm{sig}},\; C_{\mathrm{priv}} > 0
\]
such that, for all sufficiently large \(n\), if
\[
\varepsilon \ge L \log n,
\]
then
\[
e^{-C_{\mathrm{sig}}\mathrm{Signal}} + e^{-C_{\mathrm{priv}}\varepsilon}
\;\lesssim\;
\mathcal R_\varepsilon(n,K,a,b,\beta)
\;\lesssim\;
e^{-c_{\mathrm{sig}}\mathrm{Signal}} + e^{-c_{\mathrm{priv}}\varepsilon}.
\]
Equivalently,
\[
\mathcal R_\varepsilon(n,K,a,b,\beta)
=
\exp\{-\Theta(\mathrm{Signal})\}
+
\exp\{-\Theta(\varepsilon)\}.
\]
\end{corollary}

\begin{remark}
Note that the lower bound in \cite{Zhang2016Minimax} continues to hold under the additional assumption of constant SNR. For instance, it is enough to
assume that
\[
b \ge c a
\]
for some constant \(c\in(0,1)\), so that \(a \asymp b\). The point is that
the least favorable construction used in the proof keeps the same pair
\((a,b)\) and only restricts the community sizes, hence it remains admissible
under this additional assumption. Therefore the reduction from the global
misclassification risk to the corresponding local testing problem is
unchanged.
\end{remark}


\section{Conclusion and open problems}

We studied community detection in stochastic block models under pure node-level differential privacy, a stringent notion that protects the participation of a vertex together with all incident edges. On the achievability side, we analyzed an Exponential-Mechanism estimator built from the homogeneous-SBM penalized likelihood score and showed that, after restricting to a high-probability degree envelope and extending the mechanism to the full graph domain, one obtains a full-domain pure $\varepsilon$-node-DP estimator whose expected mis-match ratio is controlled by the non-private exponential-rate term together with an explicit privacy penalty. In the constant-SNR logarithmic-degree regime with fixed $K$ and $\beta$, this yields exact recovery under the requirement $\varepsilon \gtrsim \log n$, a strong improvement upon the naive $\varepsilon \gtrsim n$ requirement from directly analyzing the sensitivity of natural likelihood scores.

On the converse side, our two-point lower bound shows that pure node-DP imposes a genuine barrier for SBM recovery: any $\varepsilon$-node-DP mechanism has exact-recovery failure probability at least $(1+e^{2\varepsilon})^{-1}$ and expected mis-match at least $[n(1+e^{2\varepsilon})]^{-1}$. In particular, polynomially strong exact-recovery failure, or polynomially small expected mis-match, already forces $\varepsilon = \Omega(\log n)$. Thus the logarithmic privacy scale is not an artifact of the upper-bound analysis; it is an intrinsic feature of node-private community detection.

Taken together, our results give a clean rate-level picture of the problem. In the regime $\varepsilon \gg \log n$, the upper and lower bounds have the same two-term form, with one term governed by the SBM signal and the other by the privacy budget, matching up to universal constants in the exponents. In particular, our results identify the correct qualitative minimax tradeoff in this regime and show that exact node-private community recovery is possible without losing the exponential statistical structure of the non-private problem.

Beyond the resulting guarantees, the analysis also highlights a general mechanism for reconciling sharp likelihood-based inference with pure DP: a problem-adapted score, a $\gamma$-slack transfer from utility to risk, and a Lipschitz extension argument for controlling node sensitivity. We expect this combination of ideas to be useful more broadly for private inference problems.

\paragraph{Open problems.} Several natural questions remain.

\begin{enumerate}
    \item \textbf{Sharp constants in the node-private minimax risk.} Our upper and lower bounds already match in their two-term structure and in the relevant logarithmic privacy scale (Corollary \ref{cor:rate_cnst}). A natural next question is therefore to determine the sharp constants in the signal and privacy exponents, and more generally to pin down the exact crossover between the statistical and privacy terms as a function of $(nI,a,\varepsilon,K)$.

    \item \textbf{The threshold for vanishing exact-recovery failure.} Our lower bound implies that $\varepsilon \to \infty$ is necessary for $\Pr(r>0) \to 0$, and that polynomially small failure requires $\varepsilon = \Omega(\log n)$. Our upper bound gives sufficiency at roughly the logarithmic scale for the privacy budget, to achieve polynomially strong exact recovery in the constant-SNR logarithmic-degree regime. It remains open to determine the precise privacy threshold for (not polynomially strong) exact recovery under pure node-DP, and in particular whether $\varepsilon \asymp \log n$ is the correct threshold in all relevant regimes.

    \item \textbf{Computationally efficient node-private algorithms.} Our positive result is information-theoretic as it relies on the Exponential Mechanism over the full space $\Sigma_{\beta}$ which is not computationally efficient in general. An important next step is therefore to study whether one can design polynomial-time pure node-DP algorithms with guarantees comparable to those presented in this work, or whether node-private community detection exhibits an information-computation gap.
\end{enumerate}
\newpage

\appendix
\section{Proof of Theorem \ref{thm:private_rate}}

\subsection{A non-private baseline with $\gamma$-slack}

The first ingredient converts a near-optimality of the score \eqref{eq:T} into an exponentially small misclassification risk.

\begin{lemma}[Risk under score slack]\label{lem:risk-with-score-slack}\label{thm:risk-with-score-slack}
	Assume $K\ge 2$ and 
	\(
	\frac{nI}{K\log K}\to\infty.\)
    Let $\hat\sigma=\hat\sigma(A)\in \Sigma_\beta$ be any (possibly randomized) estimator such that, for some deterministic sequence $s_n\ge 0$,
	\begin{equation}\label{eq:score-slack-assumption}
		T_A(\hat\sigma)
		\;\ge\;
		\max_{\sigma\in\Sigma_\beta} T_A(\sigma)-s_n\qquad\text{a.s.}
	\end{equation}
	Then
	\[
	\sup_{\Theta(n,K,a,b,\beta)} \mathbb E\, r(\sigma_0,\hat\sigma)
	\;\le\;
	\begin{cases}
		\exp\!\bigl(-(1+o(1))\,\frac{nI}{2}+t^\star s_n\bigr), & K=2,\\[2mm]
		\exp\!\bigl(-(1+o(1))\,\frac{nI}{\beta K}+t^\star s_n\bigr), & K\ge 3.
	\end{cases}
	\]
\end{lemma}

\begin{proof} Let $S_K$ be the permutation group on $[K]$ and
	\[
	d(\sigma,\sigma_0):=\min_{\pi\in S_K} d_H(\sigma,\pi\circ\sigma_0).
	\]
	For $1\le m\le n-1$, define
	\begin{equation}\label{def:E_m}
	E_m(s_n)
	:=
	\Bigl\{
	\exists\,\sigma\in\Sigma_\beta:\ d(\sigma,\sigma_0)=m
	\text{ and }
	T_A(\sigma)\ge T_A(\sigma_0)-s_n
	\Bigr\}.
	\end{equation}
    
	Because $\sigma_0\in\Sigma_\beta$, \eqref{eq:score-slack-assumption} implies
	\[
	\{d(\hat\sigma,\sigma_0)=m\}\subseteq E_m(s_n).
	\]
	Hence
	\begin{equation}\label{eq:risk-layer-decomp}
		\mathbb E\,\left(r(\sigma_0,\hat\sigma)\right )
		=
		\frac1n\sum_{m=1}^{n-1} m\,\mathbb P\bigl(d(\hat\sigma,\sigma_0)=m\bigr)
		\le
		\frac1n\sum_{m=1}^{n-1} m\,\mathbb P\bigl(E_m(s_n)\bigr).
	\end{equation}
	
	\medskip
	\noindent\textbf{Case 1: $K=2$.}
	Fix $\sigma\in\Sigma_\beta$ with $d(\sigma,\sigma_0)=m$, and let
	$\alpha(\sigma;\sigma_0)$ and $\gamma(\sigma;\sigma_0)$ be the split and merge counts:
	\[
	\alpha(\sigma;\sigma_0)
	:=\big|\{i<j:\ \sigma_0(i)=\sigma_0(j),\ \sigma(i)\neq\sigma(j)\}\big|,
	\]
	\[
	\gamma(\sigma;\sigma_0)
	:=\big|\{i<j:\ \sigma_0(i)\neq\sigma_0(j),\ \sigma(i)=\sigma(j)\}\big|.
	\]
    
    By Lemma \ref{lem:bernoulli-reduction}, there exist independent random variables
	\[
	X_1,\dots,X_{\gamma}\stackrel{\mathrm{iid}}{\sim}\mathrm{Ber}(b/n),
	\qquad
	Y_1,\dots,Y_{\alpha}\stackrel{\mathrm{iid}}{\sim}\mathrm{Ber}(a/n),
	\]
	such that
	\[
	\mathbb P\!\bigl(T_A(\sigma)\ge T_A(\sigma_0)-s_n\bigr)
	\le
	\mathbb P\!\Bigl(
	\sum_{i=1}^{\gamma} X_i-\sum_{i=1}^{\alpha} Y_i
	\ge \lambda(\gamma-\alpha)-s_n
	\Bigr).
	\]
	Applying Chernoff's bound at $t^\star>0$,
	\begin{align*}
		\mathbb P\!\bigl(T_A(\sigma)\ge T_A(\sigma_0)-s_n\bigr)
		&\le
		e^{t^\star s_n}
		\bigl(\mathbb E e^{t^\star X_1}\bigr)^{\gamma}
		\bigl(\mathbb E e^{-t^\star Y_1}\bigr)^{\alpha}
		e^{-t^\star\lambda(\gamma-\alpha)}.
	\end{align*}
	By Lemma~\ref{lem:alpha-gamma-k2}, for $K=2$ we have
\[
\alpha(\sigma;\sigma_0)+\gamma(\sigma;\sigma_0)=m(n-m).
\] 
Moreover, the special choice of $\lambda$ in \eqref{eq:lambdaK2} gives
	\[
	e^{-t^\star\lambda}\mathbb E e^{t^\star X_1}
	=
	e^{t^\star\lambda}\mathbb E e^{-t^\star Y_1}
	=
	\Bigl(\mathbb E e^{t^\star X_1}\mathbb E e^{-t^\star Y_1}\Bigr)^{1/2}
	=
	e^{-I/2}.
	\]
	Therefore
	\begin{equation}\label{eq:k2-fixed-sigma}
		\mathbb P\!\bigl(T_A(\sigma)\ge T_A(\sigma_0)-s_n\bigr)
		\le
		e^{t^\star s_n}\exp\!\Bigl(-\frac{m(n-m)I}{2}\Bigr).
	\end{equation}
    By Lemma~\ref{lem:k2-layer-summation} \eqref{eq:k2-fixed-sigma}
for every $\sigma$ with $d(\sigma,\sigma_0)=m$ implies
\[
\sup_{\Theta(n,2,a,b,\beta)} \mathbb E\,r(\sigma_0,\hat\sigma)
\le
\exp\!\left\{-\Bigl(1+o(1)\Bigr)\frac{nI}{2}+t^\star s_n\right\}.
\]

	\medskip
	\noindent\textbf{Case 2: $K\ge 3$.}
	Fix $\sigma\in\Sigma_\beta$ with $d(\sigma,\sigma_0)=m$. By Lemma~\ref{lem:chernoff_comp},
	\begin{equation}\label{eq:kg3-fixed-sigma}
		\mathbb P\!\bigl(T_A(\sigma)\ge T_A(\sigma_0)-s_n\bigr)
		\le
		e^{t^\star s_n}
		\exp\!\Bigl(
		-I\bigl(\alpha(\sigma;\sigma_0)\wedge \gamma(\sigma;\sigma_0)\bigr)
		\Bigr).
	\end{equation}
    Now apply Lemma~\ref{lem:split-merge-lower-bound}. For
$m=d(\sigma,\sigma_0)$ it yields
\[
\alpha(\sigma;\sigma_0)\wedge \gamma(\sigma;\sigma_0)\ge
\begin{cases}
\dfrac{nm}{\beta K}-m^2, & m\le \dfrac{n}{2\beta K},\\[1ex]
c_\beta \dfrac{nm}{K}, & m> \dfrac{n}{2\beta K}.
\end{cases}
\]
For \(m\in\{1,\dots,n-1\}\), define
\[
q_m:=
\begin{cases}
\exp\!\left\{-I\left(\frac{nm}{\beta K}-m^2\right)\right\},
    & m\le \frac{n}{2\beta K},\\[1ex]
\exp\!\left\{-I c_\beta \frac{nm}{K}\right\},
    & m> \frac{n}{2\beta K}.
\end{cases}
\]
For $\sigma:[n]\to[K]$, define the orbit
\begin{equation}\label{def:Gamma}
    \Gamma(\sigma):=\{\pi\circ \sigma:\ \pi\in S_K\}\quad\text{and}\quad d(\Gamma,\sigma_0):=\min_{\sigma\in\Gamma} d(\sigma,\sigma_0)
\end{equation}
Since $T_A(\sigma)$ depends on $\sigma$ only through the relation
$\mathbf 1\{\sigma(i)=\sigma(j)\}$, we have
\[
T_A(\pi\circ\sigma)=T_A(\sigma)\qquad\forall \pi\in S_K,
\]
so $T_A$ is constant on each orbit.

For $m\ge 0$, let
\begin{equation}\label{def:G_m}
\mathcal G_m:=\{\Gamma(\sigma):\ \sigma\in\Sigma_\beta,\ d(\Gamma,\sigma_0)=m\}.
\end{equation}
We have that every orbit representative \(\sigma_\Gamma\) with
\(d(\sigma_\Gamma,\sigma_0)=m\) satisfies
\begin{equation}\label{eq:bound_exp_signal}
    \Pr\!\bigl(T_A(\sigma_\Gamma)\ge T_A(\sigma_0)-s_n\bigr)
\le e^{t^\star s_n} q_m
\end{equation}
which yields using Lemma \ref{lem:orbit_counting}
\[
\mathbb E r(\sigma_0,\hat\sigma)
\le \frac{e^{t^\star s_n}}{n}\sum_{m=1}^{n-1} m\, | \mathcal G_m|\, q_m.
\]
The sum on the right-hand side is exactly the non-slack layer sum analyzed in
the proof of \cite[Theorem~3.2]{Zhang2016Minimax}. Hence the same case analysis
gives
\[
\mathbb E r(\sigma_0,\hat\sigma)
\le e^{t^\star s_n}\exp\!\left\{-(1+o(1))\frac{nI}{\beta K}\right\}.
\]
Combining the two cases proves the lemma.
\end{proof}

\subsection{Sensitivity on a high-probability degree envelope}

\begin{lemma}[Degree envelope holds with high probability]\label{lem:degree-envelope}
	There exists $C>0$ large enough such that, for all $n$,
	\[
	\mathbb P\bigl(A\in\mathcal G_C\bigr) \;\ge\; 1-\exp(-10(a+\log n)).
	\]
\end{lemma}

\begin{proof}
	Fix a node $i$. Its degree is a sum of independent Bernoulli variables with means at most $a/n$ (within-community) or $b/n$ (across-community). Hence
	\[
	\mathbb E[\deg_A(i)]
	\le \frac{\beta n}{K}\cdot\frac{a}{n}
	+\Bigl(n-\frac{n}{\beta K}\Bigr)\cdot\frac{b}{n}
	\le a+b \le 2a.
	\]
	Bernstein's inequality yields, for any $t>0$,
	\[
	\mathbb P\bigl(\deg_A(i)-\mathbb E[\deg_A(i)]\ge t\bigr)
	\le
	\exp\!\left(-\frac{t^2}{4a+2t/3}\right).
	\]
	Take $t=c_1\max\{a,\log n\}$ with $c_1>0$ large enough so that the right-hand side is at most $\exp(-11(a+\log n))$ uniformly in $i$. A union bound over $i\in[n]$ gives
	\[
	d_{\max}(A)\le \mathbb E[\deg_A(i)]+t\le 2a+t\le C\max\{a,\log n\}
	\]
	with probability at least $1-\exp(-10(a+\log n))$.
\end{proof}
\begin{lemma}[Lipschitz bound on $\mathcal G_C$]
\label{lem:sensitivity}
For every $\sigma \in \Sigma_\beta$ and all $A,A' \in \mathcal G_C$,
\[
|T_A(\sigma)-T_{A'}(\sigma)| \le \Delta_a\, d_v(A,A').
\]
In particular, if $A \sim_v A'$ and $A,A' \in \mathcal G_C$, then
\[
|T_A(\sigma)-T_{A'}(\sigma)| \le \Delta_a.
\]
\end{lemma}

\begin{proof}
Write
\[
S_A(\sigma):=\sum_{i<j} A_{ij}\mathbf{1}\{\sigma(i)=\sigma(j)\}.
\]
Since the penalty term in \eqref{eq:T} does not depend on $A$, it suffices to bound
$|S_A(\sigma)-S_{A'}(\sigma)|$.

Let $U \subseteq [n]$ be such that $|U|=d_v(A,A')$ and $A_{ij}=A'_{ij}$ whenever $i,j \notin U$.
Then only edges incident to $U$ can contribute to the difference, so
\[
|S_A(\sigma)-S_{A'}(\sigma)|
\le
\sum_{v\in U}\deg_A(v)+\sum_{v\in U}\deg_{A'}(v).
\]
Since $A,A' \in \mathcal G_C$, each degree is at most $C\max\{a,\log n\}$, and therefore
\[
|S_A(\sigma)-S_{A'}(\sigma)|
\le
2|U|\,C\max\{a,\log n\}
=
\Delta_a\, d_v(A,A').
\]
This proves the claim.
\end{proof}
\subsection{EM utility via peeling}

For $s\ge 0$ define the near-optimal level set
\[
S_s(A):=\{\sigma\in\Sigma_\beta:\ T_A(\sigma)\ge \max_{\tau\in\Sigma_\beta}T_A(\tau)-s\}.
\]
When the EM is run with privacy budget $\varepsilon_0=\varepsilon/2$ and sensitivity $\Delta$, write
\[
\eta:=\frac{\varepsilon_0}{2\Delta}.
\]
The next lemmas reduce the near-optimality of $T_A$ to controlling the size of $S_s(A)$, then control the size of $S_s(A)$, and then translate that control into a high-probability utility guarantee for the EM.

\begin{lemma}[Peeling inequality]\label{lem:peeling}
	For any $s>0$ and any adjacency matrix $A$,
	\[
	\mathbb P\!\left( T_A(\hat\sigma) \le \max_{\tau \in\Sigma_\beta}T_A(\tau) - s \,\middle|\, A\right)
	\;\le\; \sum_{\ell\ge 1} |S_{\ell s}(A)|\, e^{-\eta \ell s}.
	\]
\end{lemma}

\begin{proof}
	Let $T^\star= \max_{\tau\in\Sigma_\beta} T_A(\tau)$ and partition labelings into layers
	\[
	\mathcal L_\ell(A;s):=\{\sigma \in\Sigma_\beta: T_A(\sigma)\in(T^\star-(\ell+1)s,\ T^\star-\ell s]\},
	\qquad \ell\ge 0.
	\]
	Using \eqref{eq:em}  and the bound $\sum_\tau e^{\eta T_A(\tau)}\ge e^{\eta T^\star}$,
	\[
	\mathbb P\big(T_A(\hat\sigma)\le T^\star-s\mid A\big)
	\le \sum_{\ell\ge 1}\sum_{\sigma\in \mathcal L_\ell}e^{\eta(T_A(\sigma)-T^\star)}
	\le \sum_{\ell\ge 1}|\mathcal L_\ell|\,e^{-\eta \ell s}
	\le \sum_{\ell\ge 1}|S_{\ell s}(A)|\,e^{-\eta \ell s}.
	\]
\end{proof}
\begin{lemma}[Bound on near-optimal sets]\label{lem:nearoptimal}
 Assume Assumptions~\ref{ass:strong} and~\ref{ass:mildK}.
There exist absolute constants $c_0,C_0,C_1>0$ such that, with probability at least
\[
1-\exp\{-c_0\,\mathrm{Signal}\},
\]
one has, for all $s\ge 0$,
\begin{equation}\label{eq:lem_near-optimal}
\log |S_s(A)| \le C_0 \frac{K\log(nK)}{nI}\, s + C_1 \log(nK).
\end{equation}
\end{lemma}

\begin{proof}
For $s\ge 0$, define
\[
\widetilde S_s(A):=\{\sigma\in\Sigma_\beta:\ T_A(\sigma)\ge T_A(\sigma_0)-s\}.
\]
Since $\max_{\tau\in\Sigma_\beta}T_A(\tau)\ge T_A(\sigma_0)$, we have
\[
S_s(A)\subseteq \widetilde S_s(A),
\]
so it suffices to bound $|\widetilde S_s(A)|$.

\medskip\noindent\textbf{Step 1: One-layer probability bound.}
Fix $m\ge 1$, $\Gamma\in\mathcal G_m$ where $\Gamma$ and $\mathcal G_m$ are defined in \eqref{def:Gamma} and \eqref{def:G_m}. We choose a representative $\sigma_\Gamma\in\Gamma$
such that
\[
d(\sigma_\Gamma,\sigma_0)=m.
\]
Since $T_A$ is constant on $\Gamma$,
\[
\Pr\!\left(\exists \sigma\in\Gamma:\ T_A(\sigma)\ge T_A(\sigma_0)-s\right)
=
\Pr\!\left(T_A(\sigma_\Gamma)\ge T_A(\sigma_0)-s\right).
\]
We have that there exists an absolute constant $c_{\rm test}>0$ such that
for every $s\ge 0$,
\begin{equation}\label{eq:single-layer-bound}
\Pr\!\left(T_A(\sigma_\Gamma)\ge T_A(\sigma_0)-s\right)
\le
\exp\{-c_{\rm test}\,m\,\mathrm{Signal}+t^\star s\}.
\end{equation}
Indeed, for $K\geq 3$, \eqref{eq:single-layer-bound} follows from \eqref{eq:bound_exp_signal} and 
for $K=2$, we have  $m\le n/2$ and  $n-m\ge n/2$, which implies
\[
\frac{m(n-m)I}{2}\ge \frac{mnI}{4}=\frac{m\,\mathrm{Signal}}{2}.
\]
Plugging this bound into \eqref{eq:k2-fixed-sigma}, \eqref{eq:single-layer-bound} holds with $c_{\rm test}=1/2$.

Now combine \eqref{eq:single-layer-bound} with the class count
\eqref{eq:class_count}: for every $m\ge 1$ and every $s\ge 0$,
\begin{align}
&\Pr\!\left(\exists \Gamma\in\mathcal G_m,\ \exists \sigma\in\Gamma:\ 
T_A(\sigma)\ge T_A(\sigma_0)-s\right) \nonumber\\
&\qquad\le
|\mathcal G_m|\,\exp\{-c_{\rm test}m\,\mathrm{Signal}+t^\star s\}\nonumber\\
&\qquad\le
\exp\{-\delta m+t^\star s\},
\label{eq:layer_bound_unified}
\end{align}
where
\[
\delta:=c_{\rm test}\,\mathrm{Signal}-\log(enK).
\]
By Assumption~\ref{ass:strong} $\mathrm{Signal}\ge C_s\log(nK)$ for $C_s$ large enough in both cases
($K=2$ and $K\ge 3$), so after enlarging $C_s$ if necessary we may assume
\[
\delta\ge c_\delta\,\mathrm{Signal}
\]
for some absolute constant $c_\delta>0$.

\medskip\noindent\textbf{Step 2: build one high-probability event for all $s\ge 0$.}
Fix a large absolute constant $c_\alpha>0$ and define
\[
\alpha_0:=c_\alpha \frac{\mathrm{Signal}}{\log(nK)}.
\]
For each integer $m\ge 1$, set
\[
s_m:=\max\left\{0,\ \frac{\delta m-\alpha_0\log(nK)}{t^\star}\right\}.
\]
Let
\[
\mathcal E:=
\bigcap_{m=1}^{n}
\left\{
\nexists \Gamma\in\mathcal G_m,\ \exists \sigma\in\Gamma:
\ T_A(\sigma)\ge T_A(\sigma_0)-s_m
\right\}.
\]
By \eqref{eq:layer_bound_unified} with $s=s_m$,
\[
\Pr(\mathcal E^c)\le \sum_{m=1}^{n} e^{-\delta m+t^\star s_m}.
\]
Let
\[
m_0:=\left\lceil \frac{\alpha_0\log(nK)}{\delta}\right\rceil.
\]
If $m\le m_0$, then $s_m=0$ and the summand is $e^{-\delta m}$.
If $m>m_0$, then $t^\star s_m=\delta m-\alpha_0\log(nK)$ and the summand is
$e^{-\alpha_0\log(nK)}$. Hence
\[
\Pr(\mathcal E^c)
\le
\sum_{m=1}^{m_0} e^{-\delta m}
+
\sum_{m=m_0+1}^{n} e^{-\alpha_0\log(nK)}
\le
2e^{-\delta}+n\,e^{-c_\alpha \mathrm{Signal}}.
\]
Since $\mathrm{Signal}\gtrsim \log(nK)\gtrsim \log n$, choosing $c_\alpha$ large enough yields
\[
\Pr(\mathcal E^c)\le e^{-c_0\mathrm{Signal}}
\]
for some absolute constant $c_0>0$.

\medskip\noindent\textbf{Step 3: on $\mathcal E$, all near-optimal labelings are close to the orbit of $\sigma_0$.}
Fix $s\ge 0$ and suppose $A\in\mathcal E$. Let $\sigma\in\widetilde S_s(A)$, and set
$m:=d(\sigma,\sigma_0)$. Then
\[
T_A(\sigma)\ge T_A(\sigma_0)-s,
\]
so by the definition of $\mathcal E$ we must have $s\ge s_m$. Therefore
\[
\delta m-\alpha_0\log(nK)\le t^\star s,
\]
and hence
\[
m\le m_\star(s):=
\min\left\{
n,\,
\left\lceil \frac{t^\star s+\alpha_0\log(nK)}{\delta}\right\rceil
\right\}.
\]

\medskip\noindent\textbf{Step 4: count the near-optimal labelings.}
On $\mathcal E$,
\[
|S_s(A)|\le |\widetilde S_s(A)|
\le
\sum_{m=0}^{m_\star(s)} |\{\sigma:\ d(\sigma,\sigma_0)=m\}|
\le
\sum_{m=0}^{m_\star(s)} \sum_{\Gamma\in\mathcal G_m} |\Gamma|
\le
K!\sum_{m=0}^{m_\star(s)} |\mathcal G_m|.
\]
Using \eqref{eq:class_count} and a geometric sum,
\[
|S_s(A)|
\le
K!\sum_{m=0}^{m_\star(s)} (enK)^m
\le
K!(enK)^{m_\star(s)+1}
\]
Therefore,
\[
\log |S_s(A)|
\le
\log K! + (m_\star(s)+1)\log(enK).
\]
Using the definition of $m_\star(s)$, we get
\[
\log |S_s(A)|
\le
\log K! +
\left(
\frac{t^\star s+\alpha_0\log(nK)}{\delta}+2
\right)\log(enK).
\]
Since $\delta\asymp \mathrm{Signal}$, $\alpha_0\log(nK)=c_\alpha \mathrm{Signal}$,
and $t^\star=O(1)$ under the standing constant-SNR regime, we obtain
\[
\log |S_s(A)|
\le
C\,\frac{\log(nK)}{\mathrm{Signal}}\,s + C'\log(nK)
\]
for absolute constants $C,C'>0$ ( $\log K!$ is absorbed into $C_1\log(nK)$ under Assumption~\ref{ass:mildK}).
Finally,
\[
\frac{1}{\mathrm{Signal}}
=
\begin{cases}
2/(nI), & K=2,\\[1ex]
\beta K/(nI), & K\ge 3,
\end{cases}
\]
and $\beta$ is bounded by an absolute constant in the $K\ge 3$ regime, so after
adjusting constants,
\[
\log |S_s(A)|
\le
C_0\frac{K\log(nK)}{nI}\,s + C_1\log(nK).
\]
This proves \eqref{eq:lem_near-optimal} on an event of probability at least
$1-e^{-c_0\mathrm{Signal}}$.
\end{proof}


\begin{lemma}[EM utility via peeling]\label{lem:em-utility}\label{thm:em-utility}
	Assume Assumptions~\ref{ass:strong} and~\ref{ass:mildK}. Let $\hat\sigma$ be drawn by the Exponential Mechanism \eqref{eq:em} with privacy budget $\varepsilon_0$, sensitivity $\Delta$, and inverse temperature
	\[
	\eta:=\frac{\varepsilon_0}{2\Delta}.
	\]
	Let $C_0,C_1,c_0>0$ be the constants from Lemma~\ref{lem:nearoptimal}, and define
	\[
	B := C_0\frac{K}{nI}\log(nK),
	\qquad
	\gamma_0 := \eta-B.
	\]
	Suppose $\gamma_0>0$. Then, for any $0<\alpha<\frac12$, 
    with probability at least
\[
1-\alpha-\exp\{-c_0\,\mathrm{Signal}\},
\]
over the joint draw of $A$ and $\hat\sigma$,
\[
T_A(\hat\sigma)\ge \max_{\tau\in\Sigma_\beta}T_A(\tau)-s^\star,\qquad
	s^\star := \frac{C_1\log(nK)+\log(4/\alpha)}{\gamma_0}.
\]
\end{lemma}

\begin{proof}
	By Lemma~\ref{lem:peeling}, for any $s>0$,
	\begin{equation}\label{eq:app_peeling}
		\Pr\!\big(T_A(\hat\sigma) \le \max_{\tau \in\Sigma_\beta} T_A(\tau)-s \,\big|\, A\big)
		\;\le\; \sum_{\ell\ge1} |S_{\ell s}(A)|\,e^{-\eta\,\ell s}.
	\end{equation}
	On the high-probability event of Lemma~\ref{lem:nearoptimal}, we have for all $\ell\ge1$,
	\[
	|S_{\ell s}(A)|\;\le\;\exp\!\big\{B\,\ell s + C_1\log(nK)\big\}
	\;=\; (nK)^{C_1}\,e^{B\,\ell s}.
	\]
	Hence the peeling sum in \eqref{eq:app_peeling} is bounded by a geometric series,
	\[
	\sum_{\ell\ge1} |S_{\ell s}(A)|\,e^{-\eta\,\ell s}
	\;\le\; (nK)^{C_1} \sum_{\ell\ge1} e^{-(\eta-B)\,\ell s}
	\;=\; (nK)^{C_1}\,\frac{e^{-\gamma_0 s}}{1-e^{-\gamma_0 s}},
	\qquad \gamma_0:=\eta-B>0.
	\]
	Choose $s=s^\star$ so that the right-hand side is at most $\alpha/2$.
	
	Let
	\[
	\mathcal E := \Big\{\forall s\ge 0:\ \log |S_s(A)| \le B\,s + C_1\log(nK)\Big\}.
	\]
	Then $\mathbb P(\mathcal E)\ge 1-e^{-c_0 \mathrm{Signal}}$. Taking expectations and splitting on $\mathcal E$ gives
	\begin{align*}
		\mathbb P\!\left( T_A(\hat\sigma)\le \max_{\tau\in\Sigma_\beta}T_A(\tau)-s^\star \right)
		&= \mathbb E\!\left[ \mathbb P\!\left( T_A(\hat\sigma)\le \max_{\tau\in\Sigma_\beta}T_A(\tau)-s^\star \,\middle|\, A \right)\mathbf 1_{\mathcal E} \right] \\
		&\quad + \mathbb E\!\left[ \mathbb P\!\left( T_A(\hat\sigma)\le \max_{\tau\in\Sigma_\beta}T_A(\tau)-s^\star \,\middle|\, A \right)\mathbf 1_{\mathcal E^c} \right]\\
		&\le \frac{\alpha}{2}\,\mathbb P(\mathcal E)+\mathbb P(\mathcal E^c)\\
		&\le \frac{\alpha}{2}+e^{-c_0 \mathrm{Signal}}.
	\end{align*}
	After adjusting the constant in $\log(4/\alpha)$, this yields the claimed probability bound.
\end{proof}
\subsection {Proof of Theorem~\ref{thm:private_rate}}\label{sec:proof_main_thm}

\begin{proof}
Set
\[
\varepsilon_0:=\varepsilon/2,
\qquad
\eta:=\frac{\varepsilon_0}{2\Delta_a}=\frac{\varepsilon}{4\Delta_a}.
\]

For the utility analysis, introduce the auxiliary full-domain Exponential-Mechanism sample
\[
\Pr(\bar\sigma=\sigma\mid A)
=
\frac{\exp\{\eta T_A(\sigma)\}\mathbf{1}\{\sigma\in\Sigma_\beta\}}
{\sum_{\tau\in\Sigma_\beta}\exp\{\eta T_A(\tau)\}}.
\]
For every $A\in G_C$, the conditional law of $\bar\sigma(A)$ is exactly the same as that of the
restricted-domain mechanism $\hat\sigma(A)$ in \eqref{eq:em}. Moreover, by Lemma \ref{lem:em-private},
\[
\hat\sigma^f(A)\stackrel{d}{=}\hat\sigma(A), \qquad A\in G_C.
\]
Hence, for every $A\in G_C$,
\[
\hat\sigma^f(A)\stackrel{d}{=}\bar\sigma(A).
\]

Next apply Lemma \ref{lem:em-utility} with privacy budget $\varepsilon_0$ and sensitivity $\Delta=\Delta_a$.
Then
\[
\eta=\frac{\varepsilon_0}{2\Delta_a}=\frac{\varepsilon}{4\Delta_a},
\qquad
\gamma_0=\eta-B,
\]
and, with
\[
s^\star=\frac{C_1\log(nK)+\log(4/\alpha)}{\gamma_0},
\]
we obtain
\[
\Pr\!\left(
T_A(\bar\sigma)<\max_{\tau\in\Sigma_\beta}T_A(\tau)-s^\star
\right)
\le
\alpha+e^{-c_0\mathrm{Signal}}.
\]

Write
\[
T_A^\star:=\max_{\tau\in\Sigma_\beta}T_A(\tau).
\]
Then
\[
\begin{aligned}
&\Pr\!\left(
A\in G_C,\,
T_A(\hat\sigma^f)<T_A^\star-s^\star
\right) \\
&\qquad=
E\!\left[
\mathbf{1}\{A\in G_C\}
\Pr\!\left(
T_A(\hat\sigma^f)<T_A^\star-s^\star \mid A
\right)
\right] \\
&\qquad=
E\!\left[
\mathbf{1}\{A\in G_C\}
\Pr\!\left(
T_A(\bar\sigma)<T_A^\star-s^\star \mid A
\right)
\right] \\
&\qquad\le
\Pr\!\left(
T_A(\bar\sigma)<T_A^\star-s^\star
\right)
\le
\alpha+e^{-c_0\mathrm{Signal}}.
\end{aligned}
\]

Define the event
\[
U:=\left\{
A\in G_C,\,
T_A(\hat\sigma^f)\ge T_A^\star-s^\star
\right\}.
\]
By the previous display and Lemma \ref{lem:degree-envelope},
\begin{equation}\label{eq:probability_complement}
\Pr(U^c)\le \alpha+e^{-c_0\mathrm{Signal}}+e^{-10(a+\log n)}.
\end{equation}

Let $\sigma^\star(A)$ be a measurable maximizer of $T_A$ over $\Sigma_\beta$
(for example, using a fixed tie-breaking rule), and define
\[
\tilde\sigma(A):=
\begin{cases}
\hat\sigma^f(A), & A\in U,\\
\sigma^\star(A), & A\notin U.
\end{cases}
\]
Then $\tilde\sigma\in\Sigma_\beta$ and, by construction,
\[
T_A(\tilde\sigma)\ge T_A^\star-s^\star
\qquad\text{a.s.}
\]
Hence Lemma \ref{lem:risk-with-score-slack} gives
\[
E\,r(\sigma_0,\tilde\sigma)
\le
\exp\{- (1+o(1))\mathrm{Signal}+t^\star s^\star\}.
\]

Since $\tilde\sigma(A)=\hat\sigma^f(A)$ on $U$, we have
\[
r(\sigma_0,\hat\sigma^f)
\le
r(\sigma_0,\tilde\sigma)+\mathbf{1}_{U^c}.
\]
Taking expectations and using \eqref{eq:probability_complement} yields
\[
E\,r(\sigma_0,\hat\sigma^f)
\le
\exp\{- (1+o(1))\mathrm{Signal}+t^\star s^\star\}
+\alpha+e^{-c_0\mathrm{Signal}}+e^{-10(a+\log n)}.
\]
This proves the first risk bound.

Finally, suppose $\eta\ge 2B$ and choose
\[
\alpha=(nK)^{-1}e^{-c_3\varepsilon/2}.
\]
Then
\[
\gamma_0=\eta-B\ge \eta/2=\frac{\varepsilon}{8\Delta_a},
\]
so
\[
s^\star
=
\frac{C_1\log(nK)+\log(4/\alpha)}{\gamma_0}
\lesssim
\frac{\Delta_a\log(nK)}{\varepsilon}+\Delta_a.
\]
Since $t^\star=O(1)$ under the standing constant-SNR assumption, substituting this estimate into
the previous display yields \eqref{eq:cor_simplified}.
\end{proof}


\section{Auxiliary results}\label{app:aux}
\subsection{Chernoff comparison with slack}
Next Lemma  is a direct ``with-slack'' variant of the key Chernoff comparison step used by
\cite{Zhang2016Minimax} to analyze the penalized likelihood (homogeneous-SBM MLE) score.
In their paper, for a fixed alternative labeling $\sigma$, the fundamental task is to control the
probability that $\sigma$ attains a score at least as large as the truth, i.e.
$T(\sigma)\ge T(\sigma_0)$, where $T(\cdot)$ is the penalized within-edge objective
$T(\sigma)=\sum_{i<j}(A_{ij}-\lambda)\1\{\sigma(i)=\sigma(j)\}$.
The proof proceeds by (i) decomposing the score difference into contributions over the
\emph{merge} and \emph{split} pairs (our sets $M(\sigma)$ and $S(\sigma)$, with counts
$\gamma'(\sigma;\sigma_0)$ and $\alpha(\sigma;\sigma_0)$), (ii) applying exponential Markov's
inequality and factorizing the moment generating function using conditional independence of edges,
and (iii) choosing the Chernoff tilt $t^\star$  so that the leading
mgf term is minimized and equals $\Phi(t^\star)=e^{-I}$, where $I$ is the order-$\tfrac12$
R\'enyi divergence between $\mathrm{Ber}(a/n)$ and $\mathrm{Ber}(b/n)$.
The remaining $\lambda$-dependent factors are controlled by selecting $\lambda$ in the admissible
range (equivalently, by the convex-combination parameterization of $\lambda$ in Zhang--Zhou), which
ensures these residual terms are $\le 1$ at $t=t^\star$ and yields the exponent
$\exp\{-I(\alpha\wedge \gamma')\}$ in the ``no-slack'' case.

Lemma~\ref{lem:chernoff_comp} reproduces this argument, but introduces a \emph{slack} parameter
$s\ge 0$ by bounding
\[
\Pr\big(T_A(\sigma)\ge T_A(\sigma_0)-s\big),
\]
rather than $\Pr(T_A(\sigma)\ge T_A(\sigma_0))$.
This modification is exactly what is needed for our privacy analysis: the Exponential Mechanism
typically returns a labeling whose score is near-optimal (within an additive gap) rather than
strictly optimal.

\begin{lemma}[Chernoff comparison at $t^\star$ with slack]\label{lem:chernoff_comp}
	Let $A$ be drawn from the $K$-class SBM, conditional on the ground-truth labeling
	$\sigma_0$, with within-edge probability $p:=a/n$ and across-edge probability
	$q:=b/n$. Assume
	\[
	0<q<p<1.
	\]
	Fix any labeling $\sigma:[n]\to[K]$, and define the split and merge sets
	\[
	S(\sigma):=\{(i,j): i<j,\ \sigma_0(i)=\sigma_0(j),\ \sigma(i)\neq \sigma(j)\},
	\]
	\[
	M(\sigma):=\{(i,j): i<j,\ \sigma_0(i)\neq \sigma_0(j),\ \sigma(i)=\sigma(j)\},
	\]
	with counts
	\[
	\alpha(\sigma;\sigma_0):=|S(\sigma)|,
	\qquad
	\gamma'(\sigma;\sigma_0):=|M(\sigma)|.
	\]
	Let
	\[
	t^\star:=\frac12\log\!\left(\frac{p(1-q)}{q(1-p)}\right)>0.
	\]
	Assume that $\lambda$ is chosen so that
	\begin{equation}
		\frac{1}{t^\star}\log\!\bigl(qe^{t^\star}+1-q\bigr)
		\le \lambda \le
		-\frac{1}{t^\star}\log\!\bigl(pe^{-t^\star}+1-p\bigr).
		\label{eq:A1-lambda-interval}
	\end{equation}
	Equivalently,
	\begin{equation}
		e^{-t^\star\lambda}\bigl(qe^{t^\star}+1-q\bigr)\le 1,
		\qquad
		e^{t^\star\lambda}\bigl(pe^{-t^\star}+1-p\bigr)\le 1.
		\label{eq:A1-lambda-interval-exp}
	\end{equation}
	Then, for every $s\ge 0$,
	\begin{equation}
		\mathbb{P}\!\left(T_A(\sigma)\ge T_A(\sigma_0)-s\right)
		\le
		\exp\!\left\{-I\bigl(\alpha(\sigma;\sigma_0)\wedge \gamma'(\sigma;\sigma_0)\bigr)
		+t^\star s\right\}.
		\label{eq:A1-chernoff-slack}
	\end{equation}
\end{lemma}

\begin{proof}
Note that,
for the choices of $\lambda$ in \eqref{eq:lambdaK2}--\eqref{eq:lambdaKge3}, condition \eqref{eq:A1-lambda-interval}
holds: for $K=2$, \eqref{eq:lambdaK2} is the midpoint of the interval in
\eqref{eq:A1-lambda-interval}, while for $K\ge 3$, \eqref{eq:lambdaKge3} is a convex combination
of its two endpoints.

	Write
	\[
	s_{ij}:=\mathbf{1}\{\sigma(i)=\sigma(j)\},
	\qquad
	s^0_{ij}:=\mathbf{1}\{\sigma_0(i)=\sigma_0(j)\}.
	\]
	Then
	\[
	T_A(\sigma)-T_A(\sigma_0)
	=
	\sum_{i<j}(A_{ij}-\lambda)(s_{ij}-s^0_{ij}).
	\]
	Only pairs with $s_{ij}\neq s^0_{ij}$ contribute. If $(i,j)\in M(\sigma)$,
	then $s^0_{ij}=0$ and $s_{ij}=1$, so the contribution is $A_{ij}-\lambda$.
	If $(i,j)\in S(\sigma)$, then $s^0_{ij}=1$ and $s_{ij}=0$, so the contribution
	is $\lambda-A_{ij}$. Therefore,
	\[
	Z:=T_A(\sigma)-T_A(\sigma_0)
	=
	\sum_{(i,j)\in M(\sigma)}(A_{ij}-\lambda)
	+
	\sum_{(i,j)\in S(\sigma)}(\lambda-A_{ij}).
	\]
	
	Fix $t>0$. By Markov's inequality,
	\[
	\mathbb{P}(Z\ge -s)
	=
	\mathbb{P}(e^{tZ}\ge e^{-ts})
	\le
	e^{ts}\,\mathbb{E}[e^{tZ}].
	\]
	Conditional on $\sigma_0$, the edges $\{A_{ij}:i<j\}$ are independent, and
	\[
	A_{ij}\sim \mathrm{Ber}(q)\quad\text{for }(i,j)\in M(\sigma),
	\qquad
	A_{ij}\sim \mathrm{Ber}(p)\quad\text{for }(i,j)\in S(\sigma).
	\]
	Hence
	\[
	\mathbb{E}[e^{tZ}]
	=
	\prod_{(i,j)\in M(\sigma)}\mathbb{E}[e^{t(A_{ij}-\lambda)}]
	\prod_{(i,j)\in S(\sigma)}\mathbb{E}[e^{t(\lambda-A_{ij})}]
	=
	\Bigl(e^{-t\lambda}(qe^t+1-q)\Bigr)^{\gamma'}
	\Bigl(e^{t\lambda}(pe^{-t}+1-p)\Bigr)^{\alpha},
	\]
	where, for brevity, $\alpha=\alpha(\sigma;\sigma_0)$ and
	$\gamma'=\gamma'(\sigma;\sigma_0)$.
	
	Set $m:=\alpha\wedge\gamma'$. Then
	\[
	\mathbb{E}[e^{tZ}]
	=
	\Bigl((qe^t+1-q)(pe^{-t}+1-p)\Bigr)^m
	\Bigl(e^{-t\lambda}(qe^t+1-q)\Bigr)^{\gamma'-m}
	\Bigl(e^{t\lambda}(pe^{-t}+1-p)\Bigr)^{\alpha-m}.
	\]
	Now choose $t=t^\star$. By \eqref{eq:A1-lambda-interval-exp}, the last two
	factors are at most $1$, so
	\[
	\mathbb{E}[e^{t^\star Z}]
	\le
	\Phi(t^\star)^m,
	\qquad
	\Phi(t):=(qe^t+1-q)(pe^{-t}+1-p).
	\]
	
	It remains to evaluate $\Phi(t^\star)$. Since
	\[
	e^{2t^\star}=\frac{p(1-q)}{q(1-p)},
	\]
	we have
	\[
	p(1-q)e^{-t^\star}=q(1-p)e^{t^\star}
	=\sqrt{pq(1-p)(1-q)}.
	\]
	Therefore,
	\[
	\Phi(t^\star)
	=
	pq+(1-p)(1-q)+p(1-q)e^{-t^\star}+q(1-p)e^{t^\star}
	=
	\bigl(\sqrt{pq}+\sqrt{(1-p)(1-q)}\bigr)^2.
	\]
	By the definition of the order-$\tfrac12$ R\'enyi divergence $I$,
	\[
	\Phi(t^\star)=e^{-I}.
	\]
	Combining the previous displays gives
	\[
	\mathbb{E}[e^{t^\star Z}]
	\le e^{-Im}.
	\]
	Finally,
	\[
	\mathbb{P}\!\left(T_A(\sigma)\ge T_A(\sigma_0)-s\right)
	=
	\mathbb{P}(Z\ge -s)
	\le
	e^{t^\star s}\,\mathbb{E}[e^{t^\star Z}]
	\le
	\exp\!\left\{-Im+t^\star s\right\},
	\]
	which is exactly \eqref{eq:A1-chernoff-slack}.
\end{proof}


\subsection{Proof Corollary \ref{cor:logn-necessary}}
\begin{proof}

	If $\delta(M)\le n^{-c}$, then \eqref{Eq:mismatch_pr_lb} implies
	\[
	\frac{1}{1+e^{2\varepsilon}}\ \le\ n^{-c}
	\quad\Longrightarrow\quad
	1+e^{2\varepsilon}\ \ge\ n^c
	\quad\Longrightarrow\quad
	e^{2\varepsilon}\ \ge\ n^c-1,
	\]
	hence $\varepsilon\ge \frac12\log(n^c-1)$. For $n^c\ge 2$ one has $n^c-1\ge \frac12 n^c$, giving
	$\varepsilon\ge \frac{c}{2}\log n-\frac12\log 2$.
	
	Similarly, if $R(M)\le n^{-(1+c)}$, then \eqref{eq:Er-lb} implies
	\[
	\frac{1}{n(1+e^{2\varepsilon})}\ \le\ n^{-(1+c)}
	\quad\Longrightarrow\quad
	\frac{1}{1+e^{2\varepsilon}}\ \le\ n^{-c},
	\]
	and the same algebra yields $\varepsilon\ge \frac12\log(n^c-1)$.
\end{proof}
\subsection{Reduction to a difference of Bernoulli sums}
Next lemma isolates the Bernoulli-sum reduction underlying \cite[Proposition~5.1]{Zhang2016Minimax} while Lemma \ref{lem:chernoff_comp} provides the corresponding Chernoff comparison with additive score slack.
\begin{lemma}
\label{lem:bernoulli-reduction}
Fix a ground-truth labeling $\sigma_0 \in \Sigma_\beta$ and a candidate labeling
$\sigma \in \Sigma_\beta$. Recall that
\[
T_A(\sigma)=\sum_{i<j}(A_{ij}-\lambda)\mathbf{1}\{\sigma(i)=\sigma(j)\}.
\]
Define the split and merge sets
\[
S(\sigma;\sigma_0)
:=\{(i,j): i<j,\ \sigma_0(i)=\sigma_0(j),\ \sigma(i)\neq \sigma(j)\},
\]
\[
M(\sigma;\sigma_0)
:=\{(i,j): i<j,\ \sigma_0(i)\neq \sigma_0(j),\ \sigma(i)=\sigma(j)\},
\]
and let
\[
\alpha(\sigma;\sigma_0):=|S(\sigma;\sigma_0)|,\qquad
\gamma(\sigma;\sigma_0):=|M(\sigma;\sigma_0)|.
\]
Then, under the homogeneous SBM with truth $\sigma_0$,
\[
T_A(\sigma)-T_A(\sigma_0)
\overset{d}{=}
\sum_{u=1}^{\gamma(\sigma;\sigma_0)} X_u
-
\sum_{v=1}^{\alpha(\sigma;\sigma_0)} Y_v
-\lambda\bigl(\gamma(\sigma;\sigma_0)-\alpha(\sigma;\sigma_0)\bigr),
\]
where
\[
X_1,\dots,X_{\gamma(\sigma;\sigma_0)} \stackrel{\text{iid}}{\sim} \mathrm{Ber}(b/n),
\qquad
Y_1,\dots,Y_{\alpha(\sigma;\sigma_0)} \stackrel{\text{iid}}{\sim} \mathrm{Ber}(a/n),
\]
and the two families are independent. Consequently, for every $s\ge 0$,
\[
\mathbb{P}\!\left(T_A(\sigma)\ge T_A(\sigma_0)-s\right)
=
\mathbb{P}\!\left(
\sum_{u=1}^{\gamma(\sigma;\sigma_0)} X_u
-
\sum_{v=1}^{\alpha(\sigma;\sigma_0)} Y_v
\ge
\lambda\bigl(\gamma(\sigma;\sigma_0)-\alpha(\sigma;\sigma_0)\bigr)-s
\right).
\]
\end{lemma}

\begin{proof}
Write, for brevity,
\[
\alpha=\alpha(\sigma;\sigma_0),\qquad \gamma=\gamma(\sigma;\sigma_0),\qquad
S=S(\sigma;\sigma_0),\qquad M=M(\sigma;\sigma_0).
\]
Since
\[
T_A(\sigma)-T_A(\sigma_0)
=
\sum_{i<j}(A_{ij}-\lambda)
\bigl(\mathbf{1}\{\sigma(i)=\sigma(j)\}-\mathbf{1}\{\sigma_0(i)=\sigma_0(j)\}\bigr),
\]
only pairs $(i,j)$ for which the indicators differ contribute to the difference.
If $(i,j)\in M$, then
\[
\mathbf{1}\{\sigma(i)=\sigma(j)\}-\mathbf{1}\{\sigma_0(i)=\sigma_0(j)\}=1,
\]
while if $(i,j)\in S$, then
\[
\mathbf{1}\{\sigma(i)=\sigma(j)\}-\mathbf{1}\{\sigma_0(i)=\sigma_0(j)\}=-1.
\]
Therefore
\[
T_A(\sigma)-T_A(\sigma_0)
=
\sum_{(i,j)\in M}(A_{ij}-\lambda)-\sum_{(i,j)\in S}(A_{ij}-\lambda)
=
\sum_{(i,j)\in M}A_{ij}-\sum_{(i,j)\in S}A_{ij}-\lambda(\gamma-\alpha).
\]

Now enumerate the pairs in $M$ and $S$ as
\[
M=\{(i_u,j_u):1\le u\le \gamma\},\qquad
S=\{(k_v,\ell_v):1\le v\le \alpha\},
\]
and define
\[
X_u:=A_{i_u j_u},\qquad 1\le u\le \gamma,
\qquad
Y_v:=A_{k_v \ell_v},\qquad 1\le v\le \alpha.
\]
Under the homogeneous SBM with truth $\sigma_0$, every pair in $M$ is an across-community
pair for $\sigma_0$, so
\[
X_u \sim \mathrm{Ber}(b/n),\qquad 1\le u\le \gamma.
\]
Likewise, every pair in $S$ is a within-community pair for $\sigma_0$, so
\[
Y_v \sim \mathrm{Ber}(a/n),\qquad 1\le v\le \alpha.
\]
Since the upper-triangular entries of $A$ are independent under the SBM, the collection
$\{X_u\}_{u=1}^\gamma$ is independent of the collection $\{Y_v\}_{v=1}^\alpha$, and all
variables inside each collection are independent as well. Hence
\[
T_A(\sigma)-T_A(\sigma_0)
\overset{d}{=}
\sum_{u=1}^{\gamma} X_u-\sum_{v=1}^{\alpha} Y_v-\lambda(\gamma-\alpha).
\]
The probability identity follows immediately by rearranging the event
\[
T_A(\sigma)\ge T_A(\sigma_0)-s.
\]
\end{proof}
\subsection{Two-community identity for split and merge counts}
Next we provide the proof of two-community identity for split and merge counts from Appendix A.2 in \cite{Zhang2016Minimax}.
\begin{lemma}
\label{lem:alpha-gamma-k2}
Assume $K=2$. Let $\sigma_0,\sigma:[n]\to\{1,2\}$ and suppose
\[
d(\sigma,\sigma_0)=m,
\]
where
\[
d(\sigma,\sigma_0):=\min_{\pi\in S_2} d_H(\sigma,\pi\circ\sigma_0).
\]
Define
\[
\alpha(\sigma;\sigma_0)
:=\bigl|\{(i,j): i<j,\ \sigma_0(i)=\sigma_0(j),\ \sigma(i)\neq \sigma(j)\}\bigr|,
\]
\[
\gamma(\sigma;\sigma_0)
:=\bigl|\{(i,j): i<j,\ \sigma_0(i)\neq \sigma_0(j),\ \sigma(i)=\sigma(j)\}\bigr|.
\]
Then
\[
\alpha(\sigma;\sigma_0)+\gamma(\sigma;\sigma_0)=m(n-m).
\]
\end{lemma}

\begin{proof}
Since $\alpha(\sigma;\sigma_0)$ and $\gamma(\sigma;\sigma_0)$ depend only on the partition
induced by $\sigma$, they are unchanged if we globally relabel the two classes of $\sigma$.
Hence, after composing $\sigma$ with the nontrivial permutation in $S_2$ if necessary, we may
assume
\[
d_H(\sigma,\sigma_0)=d(\sigma,\sigma_0)=m.
\]

Let
\[
D:=\{i\in[n]: \sigma(i)\neq \sigma_0(i)\},
\]
so that $|D|=m$. Since there are only two labels, for every $i\in D$ we necessarily have
\[
\sigma(i)=3-\sigma_0(i),
\]
while for every $i\notin D$ we have $\sigma(i)=\sigma_0(i)$.

Now fix any pair $i<j$. There are three cases.

\medskip
\noindent
(i) If either both $i,j\in D$ or both $i,j\notin D$, then the relation
“same label / different label” is unchanged from $\sigma_0$ to $\sigma$.
Indeed, either both labels are unchanged, or both are flipped simultaneously, and in either case
\[
\mathbf 1\{\sigma(i)=\sigma(j)\}
=
\mathbf 1\{\sigma_0(i)=\sigma_0(j)\}.
\]

\medskip
\noindent
(ii) If exactly one of $i,j$ belongs to $D$, then that relation is reversed:
\[
\mathbf 1\{\sigma(i)=\sigma(j)\}
=
1-\mathbf 1\{\sigma_0(i)=\sigma_0(j)\}.
\]

Therefore,
\[
\mathbf 1\{\sigma(i)=\sigma(j)\}\neq \mathbf 1\{\sigma_0(i)=\sigma_0(j)\}
\quad\Longleftrightarrow\quad
|\{i,j\}\cap D|=1.
\]
But the pairs for which the equality relation changes are precisely the pairs counted by
$\alpha(\sigma;\sigma_0)$ or by $\gamma(\sigma;\sigma_0)$. Hence
\[
\alpha(\sigma;\sigma_0)+\gamma(\sigma;\sigma_0)
=
\bigl|\{(i,j): i<j,\ |\{i,j\}\cap D|=1\}\bigr|.
\]
The right-hand side is just the number of unordered pairs with one endpoint in $D$ and the
other in $D^c$, namely
\[
|D|\,|D^c|=m(n-m).
\]
This proves the claim.
\end{proof}
\subsection{Two-class layer counting and summation}
The following lemma is a 
$K=2$ specialization of the layer-counting argument in Zhang and Zhou \cite[Theorem~3.1 and Appendix~A.2]{Zhang2016Minimax}, rewritten here for convenience and with the straightforward modification to accommodate an additive score slack $s_n$. 
\begin{lemma}
\label{lem:k2-layer-summation}
Assume $K=2$ and $nI\to\infty$. Let $\hat\sigma=\hat\sigma(A)\in\Sigma_\beta$
be any estimator such that
\[
T_A(\hat\sigma)\ge \max_{\sigma\in\Sigma_\beta} T_A(\sigma)-s_n
\]
for some deterministic sequence $s_n\ge 0$. Suppose moreover that for every
$\sigma\in\Sigma_\beta$ with $d(\sigma,\sigma_0)=m$,
\begin{equation}\label{eq:A1}
\mathbb P\!\left(T_A(\sigma)\ge T_A(\sigma_0)-s_n\right)
\le
\exp\!\left\{-\frac{m(n-m)}{2}I+t^\star s_n\right\}.
\end{equation}
Then
\[
\mathbb E\, r(\sigma_0,\hat\sigma)
\le
\exp\!\left\{-\Bigl(1+o(1)\Bigr)\frac{nI}{2}+t^\star s_n\right\}.
\]
\end{lemma}

\begin{proof}
For each integer $m$ with $1\leq m\leq n/2$, write
\(
P_m:=\mathbb P(E_m(s_n))
\) where \(E_m(s_n)\) is defined in \eqref{def:E_m}. We have for $K=2$ that 
\begin{equation}\label{eq:A2}
		\mathbb E\,\left(r(\sigma_0,\hat\sigma)\right )
		=
		\frac1n\sum_{m=1}^{n-1} m\,\mathbb P\bigl(d(\hat\sigma,\sigma_0)=m\bigr)
		\le
		\frac1n\sum_{m=1}^{n/2} m\,\mathbb P\bigl(E_m(s_n)\bigr).
	\end{equation}
We first bound $P_m$. Any equivalence class $\Gamma$ defined in \eqref{def:Gamma} with
$d(\Gamma,\sigma_0)=m$ admits a representative $\sigma$ with
$d_H(\sigma,\sigma_0)=m$. Such a representative is determined by choosing the
$m$ coordinates on which it differs from $\sigma_0$, and then assigning to
each chosen coordinate one of at most two labels. Therefore the number of such
equivalence classes is at most
\[
\binom{n}{m}2^m\le \left(\frac{2en}{m}\right)^m.
\]
Combining this counting bound with \eqref{eq:A1}, we obtain
\begin{equation}\label{eq:A3}
P_m\le
\left(\frac{2en}{m}\right)^m
\exp\!\left\{-\frac{m(n-m)}{2}I+t^\star s_n\right\},
\qquad 1\le m\le n/2.    
\end{equation}

Now set
\[
\widetilde P_m:=e^{-t^\star s_n}P_m.
\]
Then \eqref{eq:A3} becomes
\begin{equation}\label{eq:A4}
\widetilde P_m\le
\left(\frac{2en}{m}\right)^m
\exp\!\left\{-\frac{m(n-m)}{2}I\right\}.
\end{equation}
We treat three regimes.

\medskip
\noindent
{\bf Case 1:} There exists $0<\varepsilon<1/8$ such that
\[
\frac{nI}{2}>(1+\varepsilon)\log n
\]
for all sufficiently large $n$.
Let
\[
m_0:=1,\qquad m_1:=\frac{\varepsilon n}{2},\qquad
R_n:=n\exp\!\left\{-\frac{(n-1)I}{2}\right\}.
\]
From \eqref{eq:A4},
\[
\widetilde P_1\le R_n.
\]
For $2\le m\le m_1$, since $m\ge2$ and $n-m\ge (1-\varepsilon/2)n$,
\[
\widetilde P_m
\le
\left(\frac{2en}{2}\right)^m
\exp\!\left\{-\frac{m(n-m)}{2}I\right\}
\le
R_n\,n^{-\varepsilon m/4}
\]
for all large $n$. For $m_1<m\le n/2$, we have
\[
\widetilde P_m
\le
\left(\frac{2en}{\varepsilon n}\right)^m
\exp\!\left\{-\frac{nm}{4}I\right\}
\le
R_n\exp\!\left\{-\frac{n(m-4)}{8}I\right\}
\]
for all large $n$. Therefore
\[
\sum_{m=2}^{n/2} m\widetilde P_m=o(R_n).
\]
Using \eqref{eq:A2},
\[
\mathbb E\,r(\sigma_0,\hat\sigma)
\le
\frac{e^{t^\star s_n}}{n}\sum_{m=1}^{n/2}m\widetilde P_m
=
(1+o(1))e^{t^\star s_n}\exp\!\left\{-\frac{(n-1)I}{2}\right\}.
\]
This is
\[
\exp\!\left\{-\Bigl(1+o(1)\Bigr)\frac{nI}{2}+t^\star s_n\right\}.
\]

\medskip
\noindent
{\bf Case 2:} There exists $0<\varepsilon<1/8$ such that
\[
\frac{nI}{2}<(1-\varepsilon)\log n
\]
for all sufficiently large $n$.
Define
\[
m_0:=\Bigl\lceil n\exp\!\Bigl(-\bigl(1-e^{-\varepsilon nI/2}\bigr)\frac{nI}{2}\Bigr)\Bigr\rceil,
\qquad
m_1:=\Bigl\lceil n e^{-nI/8}\Bigr\rceil.
\]
Then $m_0\ge n^{\varepsilon/2}$ and $m_0=o(m_1)$.
Also,
\[
r(\sigma_0,\hat\sigma)
\le
\frac{m_0}{n}
+\mathbf 1\{d(\hat\sigma,\sigma_0)>m_0\},
\]
so by \eqref{eq:A2},
\begin{equation}\label{eq:A5}
\mathbb E\,r(\sigma_0,\hat\sigma)
\le
\frac{m_0}{n}
+\sum_{m>m_0} P_m
=
\frac{m_0}{n}
+e^{t^\star s_n}\sum_{m>m_0}\widetilde P_m.
\end{equation}

For $m_0<m\le m_1$, using \eqref{eq:A4}, $m\ge m_0$, and $n-m\ge n-m_1$,
\[
\widetilde P_m
\le
\left(\frac{2en}{m_0}\right)^m
\exp\!\left\{-\frac{m(n-m_1)}{2}I\right\}
\le
\exp\!\left\{-\frac14 e^{-\varepsilon nI/2}nmI\right\}
=:Q_m
\]
for all large $n$. Since $\{Q_m\}_{m_0<m\le m_1}$ is geometrically decaying and
$e^{-\varepsilon nI/2}m_0\to\infty$, we have
\[
\sum_{m_0<m\le m_1}Q_m=o(m_0/n).
\]
For $m_1<m\le n/2$, again by \eqref{eq:A4},
\[
\widetilde P_m
\le
\left(\frac{2en}{m_1}\right)^m
\exp\!\left\{-\frac{nm}{4}I\right\}
\le
\exp\!\left\{-\frac{nmI}{16}\right\}
=:Q'_m
\]
for all large $n$, and $\sum_{m>m_1}Q'_m=o(m_0/n)$.
Substituting into \eqref{eq:A5} gives
\[
\mathbb E\,r(\sigma_0,\hat\sigma)
\le
(1+o(1))\frac{m_0}{n}\,e^{t^\star s_n}.
\]
Since
\[
\frac{m_0}{n}
=
\exp\!\left\{-\Bigl(1-e^{-\varepsilon nI/2}\Bigr)\frac{nI}{2}+o(1)\right\}
=
\exp\!\left\{-\Bigl(1+o(1)\Bigr)\frac{nI}{2}\right\},
\]
the desired bound follows.

\medskip
\noindent
{\bf Case 3:} 
\[
\frac{nI}{2\log n}\to 1.
\]
Choose a positive sequence $w=w_n\to0$ such that
\[
\left|\frac{nI}{2\log n}-1\right|\ll w,
\qquad
\frac1{\sqrt{\log n}}\le w.
\]
Define
\[
m_0:=\Bigl\lceil n\exp\!\left(-(1-w)\frac{nI}{2}\right)\Bigr\rceil,
\qquad
m_1:=\lceil w^2n\rceil.
\]
Then $m_0\to\infty$ and $m_0=o(m_1)$, and again
\begin{equation}\label{eq:A6}
\mathbb E\,r(\sigma_0,\hat\sigma)
\le
\frac{m_0}{n}
+e^{t^\star s_n}\sum_{m>m_0}\widetilde P_m.
\end{equation}

For $m_0<m\le m_1$, using \eqref{eq:A4},
\[
\widetilde P_m
\le
\left(\frac{2en}{m_0}\right)^m
\exp\!\left\{-\frac{m(n-m_1)}{2}I\right\}
\le
\exp\!\left\{-\frac14\,w\,nmI\right\}
=:Q_m
\]
for all large $n$. Since $wm_0\to\infty$, the geometric sum
$\sum_{m_0<m\le m_1}Q_m=o(m_0/n)$.
For $m_1<m\le n/2$,
\[
\widetilde P_m
\le
\left(\frac{2en}{m_1}\right)^m
\exp\!\left\{-\frac{nm}{4}I\right\}
\le
\exp\!\left\{-\frac{nmI}{8}\right\}
=:Q'_m,
\]
and $\sum_{m>m_1}Q'_m=o(m_0/n)$.
Plugging these bounds into \eqref{eq:A6} yields
\[
\mathbb E\,r(\sigma_0,\hat\sigma)
\le
(1+o(1))\frac{m_0}{n}\,e^{t^\star s_n}
=
\exp\!\left\{-\Bigl(1+o(1)\Bigr)\frac{nI}{2}+t^\star s_n\right\}.
\]

This completes the proof.
\end{proof}
\subsection{Combinatorial lower bounds for split and merge counts}
The next lemma is  proven in \cite[Appendix~A.3]{Zhang2016Minimax} bounding the split and merge counts by a small-\(m\)/large-\(m\) analysis:
 \begin{lemma}[Lemma A.1 in \cite{Zhang2016Minimax}]\label{lem:split-merge-lower-bound}
Let $\sigma_0,\sigma\in \Sigma_\beta$ with $K\ge 3$, $1\le \beta < \sqrt{5/3}$, and let
\[
m:=d(\sigma,\sigma_0)=\min_{\pi\in S_K} d_H(\sigma,\pi\circ \sigma_0).
\]
Recall
\[
\alpha(\sigma;\sigma_0)
:=\bigl|\{i<j:\sigma_0(i)=\sigma_0(j),\ \sigma(i)\neq \sigma(j)\}\bigr|,
\]
\[
\gamma(\sigma;\sigma_0)
:=\bigl|\{i<j:\sigma_0(i)\neq \sigma_0(j),\ \sigma(i)=\sigma(j)\}\bigr|.
\]
Then:

\begin{enumerate}
\item If $m\le \frac{n}{2\beta K}$, then
\[
\alpha(\sigma;\sigma_0)\wedge \gamma(\sigma;\sigma_0)
\;\ge\;
\frac{nm}{\beta K}-m^2.
\]

\item There exists a constant
\[
c_\beta:=\frac{5-3\beta^2}{36\beta}>0
\]
such that whenever $m>\frac{n}{2\beta K}$,
\[
\alpha(\sigma;\sigma_0)\wedge \gamma(\sigma;\sigma_0)
\;\ge\;
c_\beta\,\frac{nm}{K}.
\]
\end{enumerate}
\end{lemma}

\subsection{Orbit counting and slack-factor propagation}
\begin{lemma}\label{lem:orbit_counting}
For \(m\in\{1,\dots,n-1\}\), let
\[
\mathcal G_m:=\{\Gamma(\sigma):\sigma\in\Sigma_\beta,\ d(\Gamma(\sigma),\sigma_0)=m\},
\qquad
\Gamma(\sigma):=\{\pi\circ \sigma:\pi\in S_K\}.
\]
Then
\begin{equation}\label{eq:class_count}
|\mathcal G_m|\le \min\Bigl\{\Bigl(\frac{enK}{m}\Bigr)^m,\ K^n\Bigr\}.
\end{equation}
Moreover, suppose that for some numbers \(q_m\ge 0\),
\[
\Pr\!\bigl(T_A(\sigma_\Gamma)\ge T_A(\sigma_0)-s_n\bigr)
   \le e^{t^\star s_n} q_m
\]
for every \(m\) and every representative \(\sigma_\Gamma\in\Gamma\) satisfying
\(d(\sigma_\Gamma,\sigma_0)=m\). Then
\[
\mathbb E\, r(\sigma_0,\hat\sigma)
   \le \frac{e^{t^\star s_n}}{n}\sum_{m=1}^{n-1} m\, |\mathcal G_m|\, q_m .
\]
\end{lemma}

\begin{proof}
Any orbit \(\Gamma\in \mathcal G_m\) admits a representative \(\sigma\) with
\(d_H(\sigma,\sigma_0)=m\). Such a representative is specified by choosing the
\(m\) vertices on which it differs from \(\sigma_0\), and then assigning to each
chosen vertex one of at most \(K\) labels. Therefore
\[
|\mathcal  G_m|\le \binom{n}{m}K^m\le \Bigl(\frac{enK}{m}\Bigr)^m.
\]
The trivial bound \(|\mathcal  G_m|\le K^n\) gives the stated minimum.

Because \(\sigma_0\in\Sigma_\beta\), we have
\[
\{d(\hat\sigma,\sigma_0)=m\}\subseteq E_m(s_n),
\]
and so
\[
\mathbb E r(\sigma_0,\hat\sigma)
\le \frac1n \sum_{m=1}^{n-1} m\, \Pr(E_m(s_n)).
\]
Now partition \(\{\sigma\in\Sigma_\beta:d(\sigma,\sigma_0)=m\}\) into the
orbits \(\mathcal  G_m\). Since \(T_A\) is constant on each orbit,
\[
\Pr(E_m(s_n))
\le \sum_{\Gamma\in G_m}
   \Pr\!\bigl(T_A(\sigma_\Gamma)\ge T_A(\sigma_0)-s_n\bigr)
\le |\mathcal G_m| e^{t^\star s_n} q_m.
\]
Substituting this into the previous display proves the claim.
\end{proof}


\section*{Acknowledgments}
 The work of O. Klopp was funded by the CY Initiative Grant Investissements d’Avenir
Agence Nationale de Recherche-16-Initiatives d’Excellence-0008 and Labex MME-DII Grant
ANR11-LBX-0023-01.

\newpage
\bibliographystyle{plain}  
\bibliography{biblio}  

\end{document}